\documentclass{amsart}

\title[Morse Index of Minimal Hypersurfaces]{An Improved Morse Index Bound of Min-Max Minimal Hypersurfaces}
\author{Yangyang Li}
\address{Department of Mathematics, Princeton University, Princeton, NJ 08544}
\email{yl15@math.princeton.edu}
\thanks{The author is partially supported by NSF-DMS-1811840.}

\usepackage{bm}
\usepackage{amsfonts,amssymb}
\usepackage{amsmath}
\usepackage{amsthm}
\usepackage{mathrsfs}
\usepackage{graphicx}
\usepackage{tikz}
\usepackage{galois}
\usepackage{soul}
\usepackage[all]{xy}
\usepackage[colorlinks, linkcolor = blue, anchorcolor = blue, citecolor = blue]{hyperref}

\newtheorem{definition}{Definition}

\newtheorem{proposition}{Proposition}
\newtheorem{corollary}{Corollary}
\newtheorem{theorem}{Theorem}

\newtheorem{remark}{Remark}
\newtheorem{lemma}{Lemma}

\def\F{\mathbf{F}}
\def\Fl{\mathcal{F}}
\def\L{\mathbf{L}}
\def\M{\mathbf{M}}
\def\R{\mathbb{R}}
\def\Zn{\mathbb{Z}}
\def\Z{\mathcal{Z}}
\def\I{\mathbf{I}}
\def\V{\mathcal{V}}
\def\IV{\mathcal{IV}}
\def\C{\mathcal{C}}
\def\A{\mathcal{A}}

\newcommand{\set}[1]{\{#1\}}

\begin{document}
\bibliographystyle{abbrvalpha}

\begin{abstract}
    In this paper, we give an improved Morse index bound of mini\-mal hypersurfaces from Almgren-Pitts min-max construction in any closed Riemannian manifold $M^{n+1}$ $(n+1 \geq 3$), which generalizes a result by X. Zhou \cite{zhou_multiplicity_2019} for $3 \leq n+1 \leq 7$. The novel techniques are the construction of hierarchical deformations and the restrictive min-max theory. These techniques do not rely on bumpy metrics, and thus could be adapted to many other min-max settings.
\end{abstract}
\maketitle

\section{Introduction}

    The Almgren-Pitts min-max construction in search of minimal hypersurfaces originated from the work of F.J. Almgren Jr. \cite{almgren_homotopy_1962, almgren_theory_1965} and J. Pitts \cite{pitts_existence_1981} (using a deep regularity result in higher dimensions by Schoen-Simon \cite{schoen_regularity_1981}), greatly developped by F.C. Marques and A. Neves \cite{marques_min-max_2014, marques_morse_2016, marques_existence_2017, marques_morse_2018-1} turns out to be a fruitful theory in solving a variety of problems in Riemannian geometry. 

    One of the most exciting applications of this theory is the confirmation of Yau's conjecture, which, roughly speaking, asserts the existence of infinitely many minimal hypersurfaces in any closed manifold. For (Baire) generic metrics, the conjecture was proved by Irie-Marques-Neves \cite{irie_density_2018} ($3 \leq n+1 \leq 7$) and the author \cite{li_existence_2019} ($n+1 \geq 8$), using Weyl law developed by Liokumovich-Marques-Neves \cite{liokumovich_weyl_2018}. In fact, the former result also revealed the denseness of minimal hypersurfaces, and later, Marques-Neves-Song \cite{marques_equidistribution_2017} gave a quantified version this, i.e., equidistribution of a sequence of minimal hypersurfaces. The general case was finally resolved by A. Song \cite{song_existence_2018} ($3\leq n+1\leq 7$) by localizing the method in \cite{marques_existence_2017} to the ``Core'' manifold of a given manifold.

    From the point of view of Morse theory, it is also important to study the geometric information, such as the Morse index and the multiplicity, of critical points of the area functional obtained from the min-max construction. The first result on the Morse index estimate was obtained by Marques-Neves \cite{marques_morse_2016} ($3 \leq n+1 \leq 7$), in which they proved the upper bound is given by the dimension of the homology class detected in the min-max process. In the same paper, they also proved the multiplicity one conjecture for $1$-parameter sweepouts. Recently, X. Zhou \cite{zhou_multiplicity_2019} confirmed the multiplicity one conjecture in multi-parameter cases ($3 \leq n+1 \leq 7$). These results together with B. Sharp's compactness theory \cite{sharp_compactness_2015} imply an improved Morse Index Bound.

    \begin{theorem}[{\cite[(improved) Theorem C]{zhou_multiplicity_2019}}]\label{thm:C}
        Given a closed manifold $M^{n+1}$ $(3 \leq n + 1 \leq 7)$, there exists a sequence of integral varifolds $V_p$ each of whose support is a disjoint union of smooth, connected, closed, embedded, minimal hypersurfaces $\set{\Sigma^p_1, \cdots, \Sigma^p_{l_p}}$ with multiplicities $\set{m^p_i}\subset \mathbb{N}^+$, such that
        \begin{equation}
            \omega_p(M, g) = \sum^{l_p}_{i = 1} m^p_i \cdot \mathrm{area}_g(\Sigma^p_i),
        \end{equation}
        is the volume spectrum with sublinear growth, and the components with multiplicity $m^p_i > 1$ are degenerately stable, while
        \begin{equation}
            \sum_{i: m^p_i = 1} \mathrm{index}(\Sigma^p_i) \leq p.
        \end{equation}
    \end{theorem}

    In the generic case, combined with the multiplicity one conjecture, Marques-Neves \cite{marques_morse_2018-1} leads to:

    \begin{theorem}[{\cite[Theorem 8.4]{marques_morse_2018-1}}]
        If $g$ is a bumpy ($C^\infty$-generic) metric on a closed manifold $M^{n+1}$ $(3 \leq n+1 \leq 7)$, then the $m^p_i$'s in the previous theorem are all $1$. Moreover, the Morse index inequality will be an identity, i.e.,
        \begin{equation}
            \sum_{i} \mathrm{index}(\Sigma^p_i) = p.
        \end{equation}
    \end{theorem}

    Note that the Morse index results above relies on the application of B. White's bumpy metrics \cite{white_space_1991} in order to bypass all the minimal hypersurfaces with large Morse index, so the lack of desirable bumpy metrics could invalidate the direct application of the existing methods. In a recent paper \cite{song_dichotomy_2019}, A. Song overcame this difficulty for $1$-parameter sweepouts in non-compact setting by constructing $2$-step deformations to establish Morse index estimates and the multiplicity one theorem. Inspired by his result, we will construct hierarchical deformations for multi-parameter sweepouts, which could be viewed as an adaption of \cite[Deformation Theorem~A]{marques_morse_2016} to the general metric setting (Theorem \ref{Thm:deform}). With these deformations, we are able to prove the Morse index bound in higher dimensions.

    \begin{theorem}[Morse Index Bound]\label{thm:main}
        Suppose that $(M^{n+1}, g)$ is a closed Riemannian manifold with $n+1 \geq 3$. Then for any $p \in \mathbb{N}^+$, there exists a stationary, integral varifold $V$ with $\mathrm{spt}(V) = \Sigma$ such that
        \begin{itemize}
            \item $\|V\|(M) = \omega_p(M, g)$;
            \item $\mathrm{index}(\Sigma) \leq p$;
            \item $\Sigma$ is a minimal hypersurface with optimal regularity, i.e., $\mathcal{H}^s(\mathrm{sing}(\Sigma)) = 0, \forall s>n-7$. In particular, when $n \leq 6$, $\Sigma$ is smooth. 
        \end{itemize}
    \end{theorem}
    \begin{remark}
        An alternative proof was obtained concurrently by A. Dey in \cite{dey_compactness_2019}, where he made a comparison of the Almgren-Pitts min-max theory and the Allen-Cahn min-max theory. Using the existing upper bound in the Allen-Cahn setting \cite{gasparSecondInnerVariation2017}, he was also able to conclude the same result in the Almgren-Pitts setting.
    \end{remark}
    \begin{remark}
        As one will see in the following, our method does not rely on the existence of bumpy metrics, or even the objects in search being minimal hypersurfaces. In principle, with a little modification, the same Morse index upper bound could be obtained for CMC hypersurfaces, PMC hypersurfaces, free-boundary minimal hypersurfaces, non-compact ambient manifold, etc.
    \end{remark}
    
    In our first Morse index bound theorem, the multiplicities do not play a role in the estimates. To capture this information, we will construct a sequence of $c$-CMC hypersurfaces to approximate a minimal hypersurface realizing the volume spectrum. Hence, we will introduce a restrictive min-max theory to make such a construction possible.

    \begin{theorem}[Improved Morse Index Bound]\label{thm:main2}
        Suppose that $(M^{n+1}, g)$ is a closed Riemannian manifold with $n+1 \geq 3$. Then for any $p \in \mathbb{N}^+$, there exists a stationary, integral varifold $V$, whose support is a disjoint union of connected minimal hypersurfaces with optimal regularity $\set{\Sigma_i}_{i=1,\cdots, l_p}$ with multiplicities $\set{m_i}$ such that
        \begin{itemize}
            \item $\|V\|(M) = \sum^{l_p}_{i = 1} m_i \cdot \mathrm{area}_g(\Sigma_i)= \omega_p(M, g)$;
            \item Every $\Sigma_i$ with $m_i \geq 3$ is stable;
            \item $\sum_{m_i \leq 2} \mathrm{index}(\Sigma_i)\leq p$.
        \end{itemize}
    \end{theorem}
    \begin{remark}
        Compared with Theorem \ref{thm:C}, the last two bullets do not seem sharp. One would expect that the multiplicity $2$ minimal hypersurfaces should also be stable as well.
    \end{remark}

    \begin{corollary}[Multiplicity $2$ for positive Ricci curvature]
        Suppose that $(M^{n+1}, g)$ $(n + 1 \geq 3)$ is a closed Riemannian manifold with positive Ricci curvature. Then for any $p \in \mathbb{N}^+$, there exists a stationary, integral varifold $V$, whose support is a connected minimal hypersurfaces with optimal regularity $\Sigma$ with multiplicities $m \leq 2$, such that
        \begin{itemize}
            \item $\|V\|(M) = m \cdot \mathrm{area}_g(\Sigma)= \omega_p(M, g)$;
            \item $\mathrm{index}(\Sigma_i)\leq p$.
        \end{itemize}
    \end{corollary}
    \begin{remark}
        When $n + 1 \leq 7$ or $p = 1$, the multiplicity has been proved to be $1$ (See \cite{zhou_multiplicity_2019}, \cite{ramirez-lunaOrientabilityMinmaxHypersurfaces2019}, \cite{bellettiniMultiplicity1MinmaxMinimal2020}). We also expect this holds in the general case as well.
    \end{remark}

    The main techniques to prove these results, as mentioned before, are the restrictive min-max theory and the construction of hierarchical deformations.

    In Section 2, after revisiting the Almgren-Pitts min-max theory and X. Zhou's $(X, Z)$-homotopy min-max theory, we introduce the restrictive min-max theory.

    In Section 3, we revisit the definition of $k$-unstability introduced by Marques-Neves \cite{marques_morse_2016}, which can be viewed as a generalization of the Morse index.

    In Section 4, we construct hierarchical deformations for multi-parameter sweepouts in the Almgren-Pitts setting and prove the Morse index bound. Roughly speaking, the deformations to perturb away minimal hypersurfaces with large Morse indices could not be constructed in one step. Instead, we construct a deformation on $0$-cells, then $1$-cells, $2$-cells and etc. The difficulty herein is restartability, which requires us to keep track of and moderate the previous deformations properly.

    In Section 5, we apply our restrictive min-max theory to construct $c$-CMC hypersurfaces approximating a $p$-width minimal hypersurface. The approximation will imply the improved Morse index bound.

\section*{Acknowledgements}
    
    The author is grateful to his advisor Fernando Cod\'a Marques for his constant support. He thanks Antoine Song for sharings his draft and his ideas which are helpful in simplifying the proof of the Deformation Theorem A. He would also thank Xin Zhou, Zhihan Wang and Akashdeep Dey for inspiring discussions.

\section{Min-max Theories}
    
    We first list some notations in geometric measure theory and Almgren-Pitts' min-max theory. Interested readers could refer to \cite{simon_lectures_1984} and \cite{pitts_existence_1981}.

    \begin{itemize}
        \item $\nu$: a metric on the flat chain space, and usually, we take $\nu = \Fl$, $\F$ or $\M$ (See the definitions in \cite{pitts_existence_1981});
        \item $\I_k(M^{n+1};\nu; \Zn_2)$: the space of $k$-dimensional modulo $2$ flat chains in $M$ with metric $\nu$. If $\nu$ is $\Fl$, we will simply put it as $\I_k(M^{n+1}; \Zn_2)$.
        \item $\Z_k(M^{n+1};\nu; \Zn_2)$: the connected component containing $0$ of the space of $k$-dimensional modulo $2$ flat cycles in $M$ with $\nu$ metric. If $\nu$ is $\Fl$,  we will simply put it as $\Z_k(M^{n+1}; \Zn_2)$;
        \item $\Z^a$: the subspace $\set{T \in \Z_n(M^{n+1}; \F; \Zn_2); \M(T) < a}$;
        \item $\IV_n(M)$: the space of $n$-dimensional integral varifolds in $M$ endowed with $\F$ metric;
        \item $\V_n(M)$: the closure of $\IV_n(M)$ in the whole space of varifolds;
        \item $|T|$: the associated integral varifold for $T \in \I_k(M^{n+1};\nu; \Zn_2)$;
        \item $\|V\|$: the associated Radon measure on $M$ for $V \in \V_n(M)$;
        \item $I(1, n)$: the cell complex on the unit interval $I$ whose $1$-cells are the intervals $[0, 1 \cdot 3^{-n}], [1 \cdot 3^{-n}, 2 \cdot 3^{-n}], \cdots, [1 - 3^{-n}, 1]$, and whose $0$-cells are the endpoints $[0], [3^{-n}], [2\cdot 3^{-n}], \cdots, [1]$;
        \item $I(m, n)$: the cell complex on $I^m$, i.e., $I(1, n)^{m\otimes} = I(1, n) \otimes I(1,n) \otimes \cdots \otimes I(1,n)$;
        \item $I(m, n)_0$: the set of $0$-cells in $I(m, n)$;
        \item $\C(M)$: the space of Cacciopoli sets in $M$, i.e., subsets of $M$ with finite perimeter.
    \end{itemize}

    \subsection{Almgren Min-Max Theory for Minimal Hypersurfaces} 

        It is well known that minimal hypersurfaces are critical points of the area functional and thus, it is natural to apply the Morse theory in the space of ``all'' closed hypersurfaces i.e., $\Z_n(M^{n+1}; \Zn_2)$, in search of these critical points.

        In 1962, Almgren in his thesis \cite{almgren_homotopy_1962} proved the following natural isomorphism:
        \begin{equation}
                \pi_k(\Z_n(M^{n+1}; \Zn_2), 0) \cong H_{k+n}(M^{n+1}; \Zn_2),
        \end{equation}
        and later, it was shown that $\Z_n(M^{n+1}; \Zn_2)$ is weakly homotopic to $\mathbb{RP}^\infty$ (\cite{marques_topology_2016}, Section 4). Hence, the space has nontrivial topology, which makes the Morse theory possible. Let's denote the generator of $H^1(\Z_n(M^{n+1};\Zn_2);\Zn_2) = \Zn_2$ by $\bar \lambda$.

        \begin{definition}[\cite{pitts_existence_1981}, \cite{li_existence_2019}]
            Given a Riemannian manifold $(M^{n+1}, g)$, for any $m\in \mathbb{N}^+$, a continuous map $\Phi : X \rightarrow \Z_n(M^{n+1}; \F; \Zn_2)$ is called a \textbf{$\bm{p}$-sweepout}, provided that $X$ is a cubic subcomplex of $I(2p+1, q)$ $(q \geq 1)$ and $\Phi^*(\bar\lambda^p) \neq 0$. In addition, the set of all $p$-sweepouts is called the \textbf{$\bm{p}$-admissible set} $\mathcal{P}_p = \mathcal{P}_p(M, g)$ on $M$.
        \end{definition}

        \begin{definition}
            The (min-max) \textbf{$\bm{p}$-width} $\omega_p(M, g)$ is
            \begin{equation}
                \omega_p(M, g) := \inf_{\Phi\in \mathcal{P}_p} \sup_x \M(\Phi(x)).
            \end{equation}
        \end{definition}

        \begin{definition}
            A \textbf{min-max sequence} of $\mathcal{P}_p$ is a sequence of $p$-sweepouts 
            \begin{equation}
                S = \set{\Phi_i}^\infty_{i = 1} \subset \mathcal{P}_p,
            \end{equation}
            with 
            \begin{equation}
                 \lim_{i \rightarrow \infty} \sup_x \M(\Phi_i(x)) = \omega_p(M, g).
            \end{equation}
            Its \textbf{critical set} is
            \begin{equation}
                \mathbf{C}(S) = \set{V \in \V_n(M)|\|V\|(M) = \omega_p(M, g) \text{ and }V = \lim_j |\Phi_{i_j}(x_j)|}.
            \end{equation}
        \end{definition}
        
        Note that the definition of $\mathcal{P}_p$ seems more restrictive than those in the literatures \cite{gromov_isoperimetry_2003}, \cite{guth_minimax_2009}, \cite{marques_existence_2017}, but in fact, one can prove that the $\omega_p$ defined here coincide with those in the literatures (See \cite{li_existence_2019}).

        \begin{theorem}[Min-max theorem for $p$-width, {\cite[Corollary 3.2]{li_existence_2019}}]
            Given a Riemannian manifold $(M^{n+1}, g)$, for each $p \in \mathbb{N}$ and any min-max sequence $S$ of $\mathcal{P}_p$, there exists a minimal hypersurfaces with optimal regularity $V \in \mathbf{C}(S)$, with the property $(2p+1)$ (See \cite[Lemma 3.2]{li_existence_2019}) and $\|V\|(M) = \omega_p(M, g)$.
        \end{theorem}
        \begin{remark}
            A minimal hypyersurface with optimal regularity here means that it is a $n$-dimensional stationary varifold whose singular set has Hausdorff codimension no less than $7$. In particular, for $3 \leq n + 1 \leq 7$, its support is a union of smooth minimal hypersurfaces.
        \end{remark}

        Here, the proper $(m)$ is taken from the proof in Pitts' combinatorial arguments \cite[Proposition~4.9]{pitts_existence_1981}, which turns out to be helpful in proving the compactness \cite{li_existence_2019}.

        \begin{definition}
            A varifold $V$ is said to have the \textbf{property $\bm{(m)}$}, provided that for any point $p \in M$ and $I_m = 5^m$ concentric annuli $\set{A_0(p, r_i - s_i, r_i + s_i)}$ where $\set{r_i}$ and $\set{s_i}$ satisfy
            \begin{align}
                r_i - 2s_i &> 10 (r_{i+1} + 2s_{i+1}), i = 1 ,\cdots, I_m - 1\\
                r_{I_m} - 2s_{I_m} &> 0,
            \end{align}
            $V$ is almost minimizing \cite[Definition~3.1]{pitts_existence_1981} in at least one of $\set{A_0(p, r_i - s_i, r_i + s_i)}$.
        \end{definition}

        \begin{definition}
            The \textbf{Almgren-Pitts Realization of $\bm{p}$-width}, denoted by $\mathcal{APR}_p$, is defined to be the set of all the varifolds $V$ satisfying
            \begin{itemize}
                \item $\|V\|(M) = \omega_p(M, g)$;
                \item $V$ is a minimal hypersurface with optimal regularity;
                \item $V$ has property $(2p + 1)$.
            \end{itemize}

            The min-max theorem above guarantees that $\mathcal{APR}_p\neq \emptyset$.
        \end{definition}

        \begin{remark}
            By \cite[Proposition~3.1]{li_existence_2019}, $\mathcal{APR}_p$ is a compact set.
        \end{remark} 

        In general, we can define an $m$-parameter $\F$-homotopy family $\Pi$ whose elements are maps with domains as a subcomplex of $I(m, k)$. 

        \begin{definition}
            We call a set $\Pi$ of continuous maps from finite dimensional simplicial complexes to $\mathcal{Z}_n(M^{n+1}; \F; \mathbb{Z}_2)$ an \textbf{$\bm{m}$-parameter $\F$-homotopy family} if the following properties hold.
            \begin{itemize}
                \item For any $\Phi \in \Pi$, $X = \mathrm{dmn}(\Phi)$ is a subcomplex of $I(m, k)$ for some $k \in \mathbb{N}^+$.
                \item For any $\Phi \in \Pi$, every continuous $\Phi': \mathrm{dmn}(\Phi) \rightarrow \mathcal{Z}_n(M^{n+1}; \F; \mathbb{Z}_2)$ homotopic to $\Phi$ in the $\F$ topology also lies in $\Pi$.
            \end{itemize}
        \end{definition}
        \begin{remark}
            Note that here we impose a stronger continuity assumption on the homotopy map than that in \cite[Definition~2.2]{li_existence_2019}.
        \end{remark}

        \begin{definition}
            The \textbf{min-max width} of $\Pi$ is
            \begin{equation}
                \L(\Pi) = \inf_{\Phi \in \Pi} \sup_{x \in \mathrm{dmn}(\Phi)}\M(\Phi(x))\,.
            \end{equation}
        \end{definition}
        
        \begin{definition}
            A \textbf{min-max sequence} of $\Pi$ is a sequence 
            \begin{equation}
                S = \set{\Phi_i}^\infty_{i = 1} \subset \Pi\,,
            \end{equation}
            with 
            \begin{equation}
                 \lim_{i \rightarrow \infty} \sup_x \M(\Phi_i(x)) = \L(\Pi)\,.
            \end{equation}

            Its \textbf{critical set} is
            \begin{equation}
                \mathbf{C}(S) = \set{V \in \V_n(M)|\|V\|(M) = \L(\Pi) \text{ and }V = \lim_j |\Phi_{i_j}(x_j)|}\,.
            \end{equation}
        \end{definition}

        \begin{theorem}[Min-max theorem for $\Pi$]
            If $\Pi$ is a $m$-parameter $\F$-homotopy family, $\L(\Pi) > 0$, and $S$ is a min-max sequence of $\Pi$, then there exists a minimal hypersurfaces with optimal regularity $V \in \mathbf{C}(S)$ with the property $(m)$ and $\|V\|(M) = \L(\Pi)$.
        \end{theorem}

        For completeness, we sketch the proof for these Min-max theorems. The details can be found in \cite{pitts_existence_1981} and \cite{marques_morse_2018-1}.

        \begin{proof}[Sketch of proof]

        Let $S = \set{\Phi_i}$. All the varifolds in the critical set $\mathbf{C}(S)$ have the same volume, i.e., the min-max width. However, some varifolds might not be stationary or with optimal regularity. Thus, the main goal is to perturb away the ``bad'' varifolds by constructing homotopic deformations.\\

        \textbf{Step 1: Pull tight.}

        For each non-stationary varifold $V \in \V_n(M)$, we can choose the dual vector field of $\delta V$ which will decrease its volume rapidly. If we ``combine'' these vector fields on $\mathcal{V}_n(M)$ using a partition of unitiy, we will obtain a continuous deformation as the induced flow on $\V_n(M)$ in the $\mathbf{F}$-topology, which can be pulled back to the cycle space with $\F$ metric under the map $T \rightarrow |T|$. For any $\Phi_i \in S$, the deformation induces a homotopy map
        \begin{equation}
            H^{(1)}_i: [0, 1] \times \mathrm{dmn}(\Phi_i) \rightarrow \Z_n(M; \F; \mathbb{Z}_2)\,,
        \end{equation}
        such that $H^{(1)}_i(0) = \Phi_i$. 

        Let $S^* = \set{\Psi_i := H^{(1)}_i(1)}$. We can check that it is a min-max sequence with $\mathbf{C}(S^*) \subset \mathbf{C}(S)$ and every $V \in \mathbf{C}(S^*)$ is stationary. It is also worthy to note that $\sup_{t, x} \M(H^{(1)}_i(t,x)) \leq \sup_{x} \M(\Phi_i(x))$.\\

        \textbf{Step 2: Discretization.}

        By \cite[Proposition~3.7]{marques_morse_2018-1}, for each $\Psi_i: X_i \rightarrow \Z_n(M; \F; \mathbb{Z}_2)$ and any $\delta_i > 0$, we can deform $\Psi_i$ to a continuous map $\Psi'_i: X_i \rightarrow \Z_n(M; \M; \mathbb{Z}_2)$ along a homotopy 
        \begin{equation}
            H^{(2)}_i: [0, 1]\times X_i \rightarrow \Z_n(M; \F; \mathbb{Z}_2)\,,
        \end{equation}
        with $H^{(2)}_i(0) = \Psi_i$ and $H^{(2)}_i(1) = \Psi'_i$, such that
        \begin{equation}
            \sup_{t, x} \F(H^{(2)}_i(t,x), \Psi_i(x)) < \delta_i.
        \end{equation}

        Thus, we always have $\sup_{t, x} \M(H^{(2)}_i(t,x)) \leq \sup_{x} \M(\Psi_i(x)) + \delta_i$. If $\delta_i \rightarrow 0$, it follows immediately that $S' := \set{\Psi'_i}$ is a min-max sequence with
        \begin{equation}
            \mathbf{C}(S') = \mathbf{C}(S^*).
        \end{equation}
        
        Following the proof of \cite[Theorem~3.8]{marques_morse_2016} (See also \cite[Lemma~3.2]{li_existence_2019}), we can find a family of discrete sweepouts to approximate $S'$. More precisely, for each $\Psi'_i \in S'$ with $X_i := \mathrm{dmn}(\Psi'_i) \subset I(m, n_i)$, we can take $N_i > n_i$ large enough and define a discrete sweepout
        \begin{equation}
             \psi_i := \Psi'_i|_{X_i \cap I(m, N_i)_0},
        \end{equation}
        such that the fineness
        \begin{equation}
            \mathbf{f}(\psi_i) := \sup_{x, y \in I(m, N_i)_0}|\psi_i(x) - \psi_i(y)| \cdot 3^{N_i} \cdot d_0(x, y) < (\delta_i)^2\,,
        \end{equation}
        where $d_0$ is the Manhattan metric.

        Now, since $\mathbf{C}(\set{\psi_i}) \subset \mathbf{C}(S')$, it suffices to show that there exists $V$ in $\mathbf{C}(\set{\psi_i})$ with optimal regularity and property ($m$).\\

        \textbf{Step 3: Interpolation.}

        We will only show the property ($m$), since the optimal regularity can be proved similarly using compactness theorems in \cite{schoen_regularity_1981}. 

        Let's argue by contradiction and suppose that there exists $\varepsilon_0 > 0$ with the following property. 

        For any $V$ in $\mathbf{C}(S)$ and any $\delta > 0$, there exists $I_m = 5^m$ concentric annuli, such that any current whose associated varifold is in $\mathbf{B}^\mathbf{F}(V, \varepsilon_0)$, has an $(\varepsilon_0, \delta)$-deformation \cite[Definition~4.1]{marques_morse_2018-1} in each concentric annuli. Roughly speaking, an $(\varepsilon_0, \delta)$-deformation means that the mass increases by at most $\delta$ along a (discrete) homotopy, but decreases by at least $\varepsilon_0$ eventually. 

        We've chosen $I_m = 5^m$ large enough with the following combinatorial property. For each $x\in X_i \cap I(m, N_i)_0$, we can associate it with an annulus in the concentric annuli for $\psi_i(x)$, such that, for any nearby $x, y$, the two annuli chosen for $\psi_i(x)$ and $\psi_i(y)$ respectively will not intersect. Thus, we could construct a deformation $h_i$ along which $\sup \mathbf{M}(\psi_i(x))$ increases by at most $(\delta_i)^2$, but eventually decreases by at least $3\varepsilon_0/4$.

        Now, it follows from \cite[Theorem 3.10]{marques_existence_2017} that for $i$ large enough, we can apply the Almgren's interpolation to $\psi_i$ and the deformation $h_i$ we just constructed, to obtain a continuous sweepout $\Psi''_i$ and a homotopy $H^{(4)}_i$, where $\Psi''_i$ can be deformed from $\Psi'_i$ via a homotopy $H^{(3)}_i$. It's worthy to note that for $i$ large enough, i.e., $\delta_i$ small enough, $H^{(3)}_i$ and $H^{(4)}_i$ will not increase $\sup_x \mathbf{M}(\Psi'_i(x))$ by more than $\delta_i$.

        In summary, when $\delta_i$ is chosen even smaller (e.g. $100\delta_i < \varepsilon_0$), we can concatenate $H^{(1)}_i$, $H^{(2)}_i$, $H^{(3)}_i$ and $H^{(4)}_i$, and obtain a new sweepout denoted by $\tilde \Psi_i$ homotopic to $\Phi_i$ with 
        \begin{equation}
            \lim_{i \rightarrow \infty} \sup_x \M(\tilde\Psi_i(x)) < \lim_{i \rightarrow \infty} \sup_x \M(\Phi_i(x))- \varepsilon_0 / 2\,,
        \end{equation}
        contradicting to the fact that $\set{\Phi_i}$ is a min-max sequence. Hence, there exits a $V$ with the property $(m)$.

        \end{proof}

    \subsection{\texorpdfstring{$(X, Z)$}{(X, Z)}-homotopy Class and \texorpdfstring{$c$}{c}-CMC Hypersurfaces}

        Here, we revisit the $(X, Z)$-homotopy class introduced by X. Zhou in \cite{zhou_multiplicity_2019}. 

        On $\C(M)$, we will define the $\F$-metric to be
        \begin{equation}
            \F(\Omega_1, \Omega_2) = \Fl(\Omega_1 - \Omega_2) + \F(|\partial \Omega_1|, |\partial \Omega_2|)\,.
        \end{equation}
    
        For a constant $c > 0$, let's define the weighted area functional $\A^c: \C(M) \rightarrow \R$ to be
        \begin{equation}
            \A^c(\Omega) = \mathcal{H}^n(\partial \Omega) - c \mathcal{H}^{n+1}(\Omega)\,.
        \end{equation}
        
        The first variational formula of $\A^c$ along a smooth vector field $Y \in \mathfrak{X}(M)$ is
        \begin{equation}
            \delta \A^c(\Omega)(Y) = \int_{\partial \Omega} \mathrm{div}_{\partial \Omega} Y \mathrm{d}\mu_{\partial \Omega} - \int_{\partial \Omega} c\langle Y, \nu\rangle \mathrm{d}\mu_{\partial \Omega},
        \end{equation}
        where $\nu = \nu_{\partial \Omega}$ is the outward unit normal on $\partial \Omega$. It follows that the smooth part of the boundary of a critical point is in fact a hypersurface with constant mean curvature $c$, i.e., $c$-CMC hypersurface.
        
        Its second variational formula for $\Omega$ along $Y = \varphi \nu$ is the same as that of the usual area functional for $\Sigma = \partial \Omega$, i.e.,
        \begin{equation}
            \delta^2 \A^c(\Omega)(Y, Y) = II_{\Sigma}(\varphi, \varphi) = \int_{\Sigma}\left(|\nabla \varphi|^2 - \left(\mathrm{Ric}^M(\nu, \nu) + |A_\Sigma|^2\right) \varphi^2 \right).
        \end{equation}
        
        \begin{definition}[{\cite[Definition~1.1]{zhou_multiplicity_2019}}] Given a pair $(X, Z)$, where $Z$ is a subcomplex of $X$, and $\Phi_0: X \rightarrow (\C(M), \F)$, we define the \textbf{$\bm{(X, Z)}$-homotopy class of $\bm{\Phi_0}$}, denoted by $\Pi(\Phi_0)$, to be the set of all sequences of maps $\set{\Phi_i: X \rightarrow (\C(M), \F)}$ satisfying the following conditions:
            \begin{enumerate}
                \item Each $\Phi_i$ is homotopic to $\Phi_0$ in the $\F$-topology on $\C(M)$;
                \item The homotopy $H_i: [0, 1] \times X \rightarrow (\C(M), \F)$ associated with $\Phi_i$ satisfies
                \begin{equation}
                    \limsup_{i \rightarrow \infty}\sup\set{\mathbf{F}(H(t,x), \Phi_0(x)): t\in [0, 1], x \in Z} = 0\,.
                \end{equation}
            \end{enumerate}
            Here, each $\set{\Phi_i}_{i \in \mathbb{N}}$ is called a \textbf{$\bm{(X, Z)}$-homotopy sequence of mappings into $\bm{\C(M)}$}.
        \end{definition}

        \begin{definition}
            The \textbf{$\bm{c}$-width} of $\Pi(\Phi_0)$ is defined by:
            \begin{equation}
                \mathbf{L}^c(\Pi(\Phi_0)) = \inf_{\set{\Phi_i} \in \Pi(\Phi_0)} \limsup_{i \rightarrow \infty} \sup_{x \in X}\set{\A^c(\Phi_i(x))}\,.
            \end{equation}
        \end{definition}
        
        \begin{definition}
            A \textbf{min-max sequence} of $\Pi(\Phi_0)$ is a sequence
            \begin{equation}
                S = \set{\Phi_i}^\infty_{i = 1} \in \Pi(\Phi_0),
            \end{equation}
            with 
            \begin{equation}
                 \lim_{i \rightarrow\infty} \mathbf{L}^c(\Phi_i) = \mathbf{L}^c(\Pi(\Phi_0)).
            \end{equation}
            Its \textbf{critical set} is
            \begin{equation}
                \mathbf{C}(S) = \set{V \in \V_n(M)|V = \lim_j |\Phi_{i_j}(x_j)| \text{ and }\lim_j \A^c(\Phi_{i_j}(x_j)) = \L^c(\Pi(\Phi_0))}.
            \end{equation}
        \end{definition}

        \begin{theorem}[Min-max Theorem for $\Pi(\Phi_0)$, {\cite[Theorem~0.3]{zhouMinmaxTheoryConstant2017}} {\cite[Theorem~0.3]{zhou_existence_2018}} {\cite[Theorem~1.7]{zhou_multiplicity_2019}} {\cite[Theorem~3.4]{deyExistenceMultipleClosed2019}}]
            Given a closed Riemannian manifold $(M^{n + 1}, g)\,(n + 1 \geq 3)$ and $c > 0$, if the $(X, Z)$-homotopy class $\Pi(\Phi_0)$ of $\Phi_0: X \rightarrow \C(M)$ satisfies
            \begin{equation}
                \L^c(\Pi(\Phi_0)) > \sup_{x \in Z} \A^c(\Phi_0(x))\,,
            \end{equation}
            then there exists a $\Omega \in \C(M)$ such that $\A^c(\Omega) = \L^c(\Pi(\Phi_0))$ and $\partial \Omega$ is a $c$-CMC hypersurface with optimal regularity.
        \end{theorem}

    \subsection{Restrictive Homotopy Class}

        In the proof of the min-max theorems (even in X. Zhou's setting), a key observation is that for $i$ large enough, along the homotopy maps $H^{(1)}_i, H^{(2)}_i, H^{(3)}_i$ and $H^{(4)}_i$, the mass will only increase by at most $3\delta_i$. In fact, the homotopy maps in Theorem \ref{Thm:deform} that we will prove later also satisfies this property. Hence, this indicates that, in the definition of homotopy classes, we can set an upper bound of the ($c$-)Area Functional along the homotopy maps and the min-max scheme still works for such a restrictive homotopy classes.

        \begin{figure}[!h]
            \centering
            \begin{tikzpicture}
            \draw[thick]  plot [smooth,tension=0.8] 
            coordinates {(-3,-1) (-1.5,2.2) (0, 0) (1.8,2.8) (3.1,-0.5)
            (3.6, 1)};

            \coordinate (O) at (0, 0);
            \coordinate (A) at (1.5, -1.5);
            \coordinate (B) at (-0.8, -0.8);
            \coordinate (T1) at (-1.5, 2.2);
            \coordinate (B1) at (-1.5, 0);
            \coordinate (T2) at (1.8, 2.8);
            \coordinate (B2) at (1.8, 0);
            \draw (A) node[below] {A} to [bend right=45] (O);
            \draw[dotted] (B) node[below] {B} to [bend left=45] (O);
            \draw[dashed, |-|] (T1) -- node[right] {$>\delta$} (B1);
            \draw[dashed, |-|] (T2) -- node[right] {$>\delta$} (B2);
            \end{tikzpicture}
            \caption{Restrictive Mountain Pass with an upper bound $\delta$}
        \end{figure}
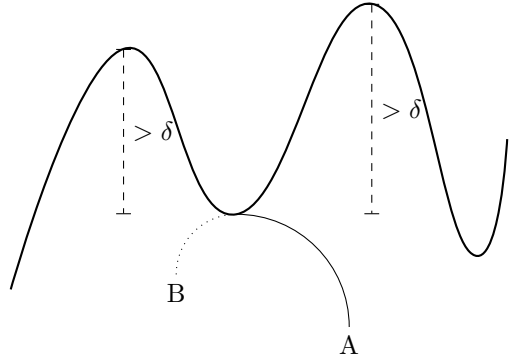

        \begin{definition}
            We call a subset $\Pi^\delta$ of continuous maps from finite dimensional simplicial complexes to $\mathcal{Z}_n(M^{n+1}; \F; \mathbb{Z}_2)$ an \textbf{restrictive $\bm{m}$-parameter $\F$-homotopy family with an upper bound $\mathbf{\delta}$} if the following properties hold.
            \begin{itemize}
                \item For any $\Phi \in \Pi^\delta$, $X = \mathrm{dmn}(\Phi)$ is a subcomplex of $I(m, k)$ for some $k \in \mathbb{N}^+$;
                \item For any $\Phi \in \Pi^\delta$, every continuous $\Phi': \mathrm{dmn}(\Phi) \rightarrow \mathcal{Z}_n(M^{n+1}; \F; \mathbb{Z}_2)$, homotopic to $\Phi$ via a homotopy map $H: X\times[0,1] \rightarrow \mathcal{Z}_n(M^{n+1}; \F; \mathbb{Z}_2)$ with $\sup_{t, x} \mathbf{M}(H(t,x)) < \sup_x \mathbf{M}(\Phi(x)) + \delta$, also lies in $\Pi^\delta$.
            \end{itemize}
        \end{definition}

        \begin{definition}
            Given a pair $(X, Z)$, where $Z$ is a subcomplex of $X$, a constant $c > 0$, and $\Phi_0: X \rightarrow (\C(M), \F)$, we define the \textbf{restrictive $\bm{(X, Z)}$-homotopy class of $\bm{\Phi_0}$ for $\A^c$ with an upper bound $\bm{\delta}$}, denoted by $\Pi^\delta_c(\Phi_0)$, to be the set of all sequences of $\mathbf{F}$-continuous maps $\set{\Phi_i: X \rightarrow (\C(M), \F)}$ satisfying the following conditions:
            \begin{enumerate}
                \item Each $\Phi_i$ is homotopic to $\Phi_0$ in the $\F$-topology on $\C(M)$, and
                \item The homotopy $H_i: [0, 1] \times X \rightarrow (\C(M), \F)$ associated with $\Phi_i$ satisfies
                    \begin{equation}
                        \limsup_{i \rightarrow \infty}\sup\set{\mathbf{F}(H(t,x), \Phi_0(x)): t\in [0, 1], x \in Z} = 0\,,
                    \end{equation}
                    and
                    \begin{equation}
                        \sup_{t, x} \A^c(H(t,x)) < \sup_{x} \A^c(\Phi_0(x)) + \delta\,.
                    \end{equation}
                    
            \end{enumerate}
            Here, each $\set{\Phi_i}_{i \in \mathbb{N}}$ is called a \textbf{restrictive $\bm{(X, Z)}$-homotopy sequence of mappings into $\bm{\C(M)}$}.
        \end{definition}

        Similarly, one can define the widths $\mathbf{L}(\Pi^\delta)$, $\mathbf{L}^c(\Pi^\delta_c(\Phi_0))$, min-max sequences, critical sets, etc. And the min-max theorems in this new setting can be proved analogously.

        \begin{theorem}[Min-max theorem for $\Pi^\delta$]
            If $\Pi^\delta$ is a $m$-parameter $\F$-homotopy family with an upper bound $\delta$, $\L(\Pi^\delta) > 0$ and $S$ is a min-max sequence of $\Pi^\delta$, there exists a minimal hypersurfaces with optimal regularity $V \in \mathbf{C}(S)$ with the property $(m)$ and $\|V\|(M) = \L(\Pi^\delta)$.
        \end{theorem}

        \begin{theorem}[Min-max Theorem for $\Pi^\delta_c(\Phi_0)$]\label{thm:rest_c}
            Given a closed Riemannian manifold $(M^{n + 1}, g)\,(n \geq 2)$ and $c > 0$ and $\delta > 0$, if the restrictive $(X, Z)$-homotopy class $\Pi^\delta_c(\Phi_0)$ of $\Phi_0: X \rightarrow \C(M)$ for $\A^c$ with an upper bound $\delta$ satisfies
            \begin{equation}
                \L^c(\Pi^\delta_c(\Phi_0)) > \sup_{x \in Z} \A^c(\Phi_0(x))\,,
            \end{equation}
            then there exists a $\Omega \in \C(M)$ such that $\A^c(\Omega) = \L^c(\Pi^\delta_c(\Phi_0))$ and $\partial \Omega$ is a $c$-CMC hypersurface with optimal regularity.
        \end{theorem}

        \begin{remark}
            To prove the Morse inequalities, the authors of \cite{marquesMorseInequalitiesArea2020} consider the homology min-max for $H_k(\Z^b, \Z^a)$, where the homotopy class $\Pi$ constructed therein is essentiall the same as the restrictive homotopy class.
        \end{remark}

\section{\texorpdfstring{$k$}{k}-unstability and Morse Index}

    In the following, $(M^{n+1}, g)$ denotes a smooth and closed Riemannian manifold and $\omega_p = \omega_p(M, g)$ for convenience.

    \subsection{\texorpdfstring{$k$}{k}-unstability}

    \begin{definition}[\cite{marques_morse_2016}]
        A stationary varifold $\Sigma \in \V_n(M)$ is called \textbf{$\bm{k}$-unstable}, provided that there exists $\varepsilon > 0$, $0 < c_0 < 1$ and a smooth family $\set{F_v}_{v\in \overline{B}^k} \subset \mathrm{Diff}(M)$ with 
        \begin{equation}
             F_0 = \mathrm{Id}, F_{-v} = F^{-1}_v\,,
        \end{equation}
        for all $v \in \overline{B}^k$ (a unit ball) such that, for any $V \in \overline{\mathbf{B}}^\F_{2\varepsilon}(\Sigma)$, the smooth function
        \begin{equation}
            A^V: \overline{B}^k \rightarrow [0, \infty), \qquad v \mapsto \|(F_v)_\# V\|(M)\,,
        \end{equation}
        satisfies
        \begin{enumerate}
            \item $A^V$ has a unique maximum at $m(V)\in B^k_{c_0 / \sqrt{10}}(0)$;
            \item $-\frac{1}{c_0} \mathrm{Id} \leq D^2 A^V(u) \leq -c_0 \mathrm{Id}$ for all $u \in \overline{B}^k$.
        \end{enumerate}
    \end{definition}

    \begin{lemma}\label{lem:tech}
        Given a $k$-unstable $\Sigma$ with a quadruple $(\varepsilon, c_0, \set{F_v}, m)$, there exists $\varepsilon_0 \in (0, \varepsilon)$ small enough, such that for any $\tilde \varepsilon \in (0, \varepsilon_0)$ and $V \in \overline{\mathbf{B}}^\F_{2\tilde\varepsilon}(\Sigma)$,
        \begin{equation}
            (F_{S^{k-1}})_{\#}(V) \cap \overline{\mathbf{B}}^\F_{2\tilde\varepsilon}(\Sigma) = \emptyset.
        \end{equation}
    \end{lemma}
    \begin{proof}
        Suppose not and there exists $V_i \in \overline{\mathbf{B}}^\F_{2\varepsilon}(\Sigma)$, $\varepsilon_i \rightarrow 0$, and $v_i \in S^k$, such that
        \begin{enumerate}
            \item $\mathbf{F}(V_i, \Sigma) \leq 2 \varepsilon_i$.
            \item $(F_{v_i})_\#(V_i) \in \overline{\mathbf{B}}^\F_{2\varepsilon_i}(\Sigma)$.
        \end{enumerate}

        By the compactness of $S^k$, up to a subsequence, we may assume that $v_i \rightarrow v \in S^k$. Since $F_v$ is smooth, we can take a limit as $i \rightarrow \infty$ and obtain
        \begin{equation}
            (F_v)_\#(\Sigma) = \Sigma\,.
        \end{equation}

        However, by calculating the area of both side, we have the following inequality
        \begin{equation}
            \begin{aligned}
                \|\Sigma\|(M) - \|(F_v)_\#(\Sigma)\|(M) &= A^\Sigma(0) - A^\Sigma(v)\\
                    &= - \int^1_0 \frac{\mathrm{d}}{\mathrm{d}t} A^\Sigma(tv) \mathrm{d}t\\
                    &= -\int^1_0 \int^t_0 \frac{\mathrm{d}^2}{\mathrm{d}^2 s} A^\Sigma(sv) \mathrm{d}s \mathrm{d}t\\
                    &= -\int^1_0 \int^t_0 D^2 A^\Sigma(sv)(v, v) \mathrm{d}s \mathrm{d}t\\
                    &\geq \int^1_0 \int^t_0 c_0 \mathrm{d}s \mathrm{d}t\\
                    &= \frac{c_0}{2} > 0\,,
            \end{aligned}
        \end{equation}
        which implies that $(F_v)_\#(\Sigma) \neq \Sigma$, giving a contradiction.
    \end{proof}

    In the following of this paper, for technical reason, we always take $\varepsilon = \varepsilon_0$ in the definition of $k$-unstability, and it is easy to check that the varifold is still $k$-unstable with the new $\varepsilon$.

    \begin{lemma}\label{lem:grad}
        Given a $k$-unstable varifold $\Sigma$ with a quadruple $(\varepsilon, c_0, \set{F_v}, m)$ and any fixed $V \in \overline{\mathbf{B}}^\F_{2\varepsilon}(\Sigma)$, if an integral curve $x: [0, L] \rightarrow \overline{B}^k$ disjoint from $m(V)$ is defined by
        \begin{equation}
            \frac{\mathrm{d}}{\mathrm{d}t}x(t) = -f(t)\nabla A^V(x(t))\,,
        \end{equation}
        where $f(t)$ is a positive continuous function, then we have
        \begin{equation}
            A^V(x(L)) - A^V(x(0)) \leq -\frac{c_0}{2}|x(L) - x(0)|^2 
        \end{equation}
    \end{lemma}
    \begin{proof}
        By the fundamental theorem of calculus, we have
        \begin{equation}
            \begin{aligned}
                &\quad A^V(x(L)) - A^V(x(0)) \\
                = &\int^L_0 \nabla A^V(x(t))\cdot \frac{\mathrm{d}}{\mathrm{d}t}x(t)\mathrm{d}t\\
                = &-\int^L_0 f(t)|\nabla A^V(x(t))|^2\mathrm{d}t\\
                \leq &-\int^L_0 f(t) |\nabla A^V(x(t))| \left(\int^t_0 \frac{\mathrm{d}}{\mathrm{d}s}|\nabla A^V(x(s))|\mathrm{d}s\right) \mathrm{d}t\\
                = &-\int^L_0 f(t) |\nabla A^V(x(t))| \int^t_0 \frac{\nabla^2 A^V(x(s)) \left(\nabla A^V(x(s)), \frac{\mathrm{d}}{\mathrm{d}s} x(s)\right)}{|\nabla A^V(x(s))|} \mathrm{d}s\mathrm{d}t\\
                = & \int^L_0 f(t) |\nabla A^V(x(t))| \int^t_0 f(s)\frac{\nabla^2 A^V(x(s)) \left(\nabla A^V(x(s)), \nabla A^V(x(s))\right)}{|\nabla A^V(x(s))|} \mathrm{d}s\mathrm{d}t\\
                \leq &- c_0 \int^L_0 f(t) |\nabla A^V(x(t))| \left(\int^t_0 f(s)|\nabla A^V(x(s))| \mathrm{d}s\right) \mathrm{d}t\,.
            \end{aligned}
        \end{equation}
        Here, we use the fact that $\nabla A^V(v) \neq 0$ for any $v \neq m(V)$. Indeed, this follows from
        \begin{equation}
            \begin{aligned}
                \left\langle\nabla A^V(v), v - m(V)\right\rangle &= \int^1_0 \frac{\mathrm{d}}{\mathrm{d}t} \left\langle\nabla A^V(v), v - m(V)\right\rangle \mathrm{d}t\\
                    &= \int^1_0 \nabla^2 A^V(t(v - m(V)) + m(V))\left(v - m(V), v - m(V)\right) \mathrm{d}t\\
                    &\leq -c_0 |v - m(V)|^2 < 0\,.
            \end{aligned}
        \end{equation}
        
        Now, let's consider the length function $l$ satisfying
        \begin{equation}
            l'(t) = \left|\frac{d}{dt}x(t)\right| = f(t)|\nabla A^V(x(t))|\,,
        \end{equation}
        with $l(0) = 0$. Apparently, the length of the curve $l(L)$ is no less than $|x(L) - x(0)|$.

        Therefore, we have
        \begin{equation}
            \begin{aligned}
                 &\quad A^V(x(L)) - A^V(x(0)) \\
                \leq& -c_0 \int^L_0 l'(t) \int^{t}_0 l'(s)\mathrm{d}u \mathrm{d}t\\
                =& -c_0 \int^L_0 l'(t) l(t) \mathrm{d}t\\
                =& -\frac{c_0}{2}l(L)^2\\
                \leq &  -\frac{c_0}{2}|x(L) - x(0)|^2 
            \end{aligned}
        \end{equation}
    \end{proof}

    \begin{lemma}\label{lem:dist_F}
        For every $k$-unstable varifold $\Sigma \in \mathcal{V}_n(M)$, there exists an increasing function
        \begin{equation}
            h_\Sigma: \mathbb{R}^+ \rightarrow \mathbb{R}^+,
        \end{equation}
        such that for any $V \in \overline{\mathbf{B}}^\F_{2 \varepsilon}(\Sigma)$, and $v, w \in \overline{B}^{k}$,
        \begin{equation}
            |v - w| \geq h_\Sigma(\F((F^\Sigma_v)_\#(V), (F^\Sigma_w)_\#(V)))\,.
        \end{equation}
    \end{lemma}
    \begin{proof}
        It suffices to show that for every $c > 0$, if
        \begin{equation}
            \set{(v, w)|\exists V \in \overline{\mathbf{B}}^\F_{2 \varepsilon}(\Sigma), \F((F^\Sigma_v)_\#(V), (F^\Sigma_w)_\#(V)) = c} \neq \emptyset\,,
        \end{equation}
        then
        \begin{equation}
            \set{|v - w||\exists V \in \overline{\mathbf{B}}^\F_{2 \varepsilon}(\Sigma), \F((F^\Sigma_v)_\#(V), (F^\Sigma_w)_\#(V)) = c}\,,
        \end{equation}
        has a positive lower bound.

        Let's argue by contradiction. Suppose that there exists a sequence $V_i \in \overline{\mathbf{B}}^\F_{2 \varepsilon}(\Sigma)$ and $v_i, w_i \in \overline{B}^k$ with $\F((F^\Sigma_{v_i})_\#(V_i), (F^\Sigma_{w_i})_\#(V_i)) = c$ but $|v_i - w_i| \rightarrow 0$. By compactness, up to a subsequence, we may assume that
        \begin{equation}
            V_i \rightarrow V, v_i \rightarrow v, w_i \rightarrow v\,,
        \end{equation}
        and $\F((F^\Sigma_{v})_\#(V), (F^\Sigma_{v})_\#(V)) = c$ gives a contradiction.
    \end{proof}

    \begin{definition}
        For a stationary varifold $\Sigma \in \V_n(M)$, we define its \textbf{index} to be
        \begin{equation}
            \mathrm{index}(\Sigma) = k,
        \end{equation}
        provided that it is $k$-unstable but not $(k+1)$-unstable. In particular, $\Sigma$ is called \textbf{stable} if it is not $1$-unstable.
    \end{definition}
    \begin{remark}
        This definition is compatible with the usual definition of the Morse index for a smooth minimal hypersurface (See \cite{marques_morse_2016}).
    \end{remark}

    Similarly, we can also adopt the weak index notion to critical points of the $\A^c$ functional.

    \begin{definition}[{\cite[(modified) Definition~2.1]{zhou_multiplicity_2019}}]
        A critical point $\Omega \in \C(M)$ of $\A^c$ is called \textbf{$\bm{k}$-unstable}, provided that there exists $\varepsilon > 0$, $0 < c_0 < 1$ and a smooth family $\set{F_v}_{v\in \overline{B}^k} \subset \mathrm{Diff}(M)$ with $F_0 = \mathrm{Id}, F_{-v} = F^{-1}_v$ for all $v \in \overline{B}^k$ such that, for any $\Omega' \in \overline{\mathbf{B}}^\F_{2\varepsilon}(\Omega)$, the smooth function
        \begin{equation}
            A^{\Omega'}_c: \overline{B}^k \rightarrow [0, \infty), \qquad A^{\Omega'}_c(v) = \A^c(F_v (\Omega')),
        \end{equation}
        satisfies
        \begin{enumerate}
            \item $A^{\Omega'}_c$ has a unique maximum at $m(\Omega')\in B^k_{c_0 / \sqrt{10}}(0)$;
            \item $-\frac{1}{c_0} \mathrm{Id} \leq D^2 A^{\Omega'}_c(u) \leq -c_0 \mathrm{Id}$ for all $u \in \overline{B}^k$.
        \end{enumerate}
        In addition, by taking $\varepsilon$ small enough, we also assume that for any $\Omega' \in \overline{\mathbf{B}}^\F_{2\varepsilon}(\Omega)$,
        \begin{equation}
            (F_{S^{k-1}})_\#(\Omega') \cap \overline{\mathbf{B}}^\F_{2 \varepsilon}(\Omega) = \emptyset.
        \end{equation}
    \end{definition}

    \begin{definition}
        For a critical point $\Omega \in \C(M)$ of $\A^c$, we define its \textbf{index} to be 
        \begin{equation}
            \mathrm{index}(\Omega) = k,    
        \end{equation}
        provided that $\Omega$ is $k$-unstable but not $(k+1)$-unstable. In particular, $\Omega$ is called \textbf{stable} if it is not $1$-unstable.
    \end{definition}

    \subsection{Almgren-Pitts realizations and Morse index}

    \begin{definition}
        Let $\bm{\mathcal{U}(\omega_p)}$ be the subset of $\mathcal{APR}_p$ consisting of all the $(p+1)$-unstable minimal hypersurfaces, and $\bm{\mathcal{S}(\omega_p)}$ the complement of $\mathcal{U}(\omega_p)$ in $\mathcal{APR}_p$, which is the subset of minimal hypersurfaces with a Morse index upper bound $p$.
    \end{definition}

    \begin{remark}
        If $\mathcal{U}(\omega_p) = \emptyset$, the main theorems are trivially true. Thus, we may assume that $\mathcal{U}(\omega_p)\neq \emptyset$.
    \end{remark}

    \begin{lemma}\label{lem:union_compact}
        $\mathcal{U}(\omega_p)$ is a countable union of compact sets, i.e.,
        \begin{equation}
            \mathcal{U}(\omega_p) = \bigcup^\infty_{i = 1}\mathcal{U}_i(\omega_p)\,,
        \end{equation}
        where $\mathcal{U}_i(\omega_p)$ is compact.
    \end{lemma}
    \begin{proof}
        By \cite[Proposition~3.1]{li_existence_2019}, $\mathcal{APR}_p$ is a compact set. By the compactness result in \cite{dey_compactness_2019}, we also know that $\mathcal{S}(\omega_p)$is compact. 
    
        If $\mathcal{S}(\omega_p) = \emptyset$, then define $\mathcal{U}_i(\omega_p) := \mathcal{U}(\omega_p)$ for all $i$, which is compact. 
    
        If $\mathcal{S}(\omega_p) \neq \emptyset$, we can define 
        \begin{equation}
            \mathcal{U}_i(\omega_p) := \set{V \in \mathcal{U}(\omega_p)| \F(V, \mathcal{S}(\omega_p)) \geq \frac{1}{i}}\,.
        \end{equation}
        Since the $\mathbf{F}$-topology is equivalent to the varifold topology on $\mathcal{APR}_p$, $\mathcal{U}_i(\omega_p)$ is a compact set and $\bigcup \mathcal{U}_i(\omega_p)$ is exactly the complement of $\mathcal{S}(\omega_p)$.

        In summary, we have
        \begin{equation}
            \mathcal{U}(\omega_p) = \bigcup^\infty_{i = 1} \mathcal{U}_i(\omega_p)\,,
        \end{equation}
        which leads to the conclusion.
    \end{proof}

\section{Deformation Theorem and Morse Index Upper Bound}

    Firstly, let's introduce some notations. Fix a sequence of varifolds $\set{\Sigma_k}^\infty_{k=1} \subset \mathcal{U}(\omega_p)$, and then by the definition of $(p + 1)$-unstability, each one is associated with a quadruple $(\varepsilon_k, c_{0,k}, \set{F^k_v}, m_k)$. 

    For a nonempty set $K\subset \mathbb{N}^+$, we define
    \begin{equation}
        \mathbf{B}^\F_K := \mathbf{B}^\F_K(\set{\Sigma_k}, \set{\varepsilon_k}) = \bigcap_{k \in K} \mathbf{B}^\F_{\varepsilon_k}(\Sigma_k)\,,
    \end{equation}
    and
    \begin{equation}
        \bar{k}(K) := \min\set{k \in K}\,.
    \end{equation}

    For $\lambda \geq 0$, we denote by $\mathbf{B}^\F_{\lambda,K}$ the set $\bigcap_{k \in K} \mathbf{B}^\F_{(2-2^{-\lambda})\varepsilon_k}(\Sigma_k)$ . In particular, we have for any $0 \leq \lambda_1 \leq \lambda_2$,
    \begin{equation}
        \mathbf{B}^\F_K = \mathbf{B}^\F_{0, K} \subset \mathbf{B}^\F_{\lambda_1, K} \subset \mathbf{B}^\F_{\lambda_2, K} \subset \bigcap_{k \in K} \mathbf{B}^\F_{2\varepsilon_k}(\Sigma_k)\,.
    \end{equation}

    For a nonempty finite subset $K \subset \mathbb{N}$, $\lambda \geq 0$ (assume $\mathbf{B}^\F_{\lambda,K} \neq \emptyset$), for each $V \in \overline{\mathbf{B}}^\F_{\lambda,K} := \bigcap_{\lambda' > \lambda} (\mathbf{B}^\F_{\lambda',K})$, we define the associated gradient flow $\set{\phi^V_{\lambda, K}(\cdot, t)}_{t\geq 0} \subset \mathrm{Diff}(\overline{B}^{p+1})$ generated by the vector field
        \begin{equation}
            u \mapsto - \F\left((F^{\bar{k}(K)}_{u})_{\#}(V), (\mathbf{B}^\F_{\lambda + 1, K})^c\right)\nabla A^V_{\bar k(K)}(u)\,,
        \end{equation} 
    where $A_{\bar k(K)}^V(u) = \|(F^{\bar k(K)}_{u})_{\#}(V)\|(M)$ for any $u \in \overline{B}^{p+1}$.

    \begin{lemma}\label{lem:mod}
        Given any $\lambda > 0$, a nonempty finite set $K \subset \mathbb{N}^+$ and small $\eta > 0$, there exist $c(\lambda, K) > 0$ and $T = T(\eta, K, \lambda) > 0$ satisfying the following property.

        For any $V \in \overline{\mathbf{B}}^\F_{\lambda,K}$ and any $v \in \overline{B}^{p+1}$ with $|v - m_{\bar{k}(K)}(V)|\geq \eta$ and
        \begin{equation}
             (F^{\bar k(K)}_{v})_\#(V) \in \overline{\mathbf{B}}^\F_{\lambda + 0.5, K},
        \end{equation}
        then we have
        \begin{equation}
            A^V_{\bar k(K)}(\phi^V_{\lambda, K}(v, T)) < A^V_{\bar k(K)}(\phi^V_{\lambda, K}(v, 0)) - c(\lambda, K).
        \end{equation}
    \end{lemma}
    \begin{proof}
        If $\overline{\mathbf{B}}^\F_{\lambda,K} = \emptyset$, then the conclusion holds trivially, so in the following, we assume that $\overline{\mathbf{B}}^\F_{\lambda,K} \neq \emptyset$.

        For each $V \in \overline{\mathbf{B}}^\F_{\lambda,K}$, we define
        \begin{equation}
            D_{small}(V) := \set{v \in \overline{B}^{p + 1}| (F^{\bar k(K)}_{v})_{\#} (V) \in \overline{\mathbf{B}}^\F_{\lambda + 0.5, K}}\,,
        \end{equation}
        and
        \begin{equation}
            D_{large}(V) := \set{v \in \overline{B}^{p + 1}| (F^{\bar k(K)}_{v})_{\#} (V) \in \overline{\mathbf{B}}^\F_{\lambda + 1, K}}\,.
        \end{equation}
        We define $d(V)$ to be the Euclidean distance between $D_{small}(V)$ and $\partial D_{large}(V)$,
        \begin{equation}
            d(V) := \mathrm{dist}(D_{small}(V), \partial D_{large}(V)).
        \end{equation}

        \textbf{Claim 1:} $d(V) > 0$.

        Indeed, by the continuity of $(F^{\bar k(K)}_\cdot)_\#(V)$, if $w \in \partial D_{large}(V)$, then either $w \in S^p$, or $(F^{\bar k(K)}_{v})_{\#} (V) \in \partial \overline{\mathbf{B}}^\F_{\lambda + 1, K}$. However, the former case is impossible due to our technical assumption (Lemma \ref{lem:tech}). In the latter case, by the continuity of $(F^{\bar k(K)}_v)_\#(V)$ again and the fact that $\partial \overline{\mathbf{B}}^\F_{\lambda + 1, K} \cap \overline{\mathbf{B}}^\F_{\lambda + 0.5, K} = \emptyset$, we can conclude that $w \notin D_{small}(V)$. In other words, $\partial D_{large}(V) \cap D_{small}(V) = \emptyset$, and thus $d(V) > 0$ since these two sets are compact. $\hfill\blacksquare$\\

        \textbf{Claim 2:} $d(\cdot)$ has a uniform positive lower bound $\tilde d = \tilde d(\lambda, K) > 0$ on $\overline{\mathbf{B}}^\F_{\lambda,K}$. 

        Suppose not, and then there exists a sequence ${V_i} \subset \overline{\mathbf{B}}^\F_{\lambda,K}$, such that $d(V_i) \rightarrow 0$. $\overline{\mathbf{B}}^\F_{\lambda,K}$ is a compact set, so we may assume that $V_i \rightarrow V \in \overline{\mathbf{B}}^\F_{\lambda,K}$. Since $\partial D_{large}(V_i)$ and $D_{small}(V_i)$ are compact, there exist  $v^1_i \in D_{small}(V_i)$ and $v^2_i \in \partial D_{larget}(V_i)$ such that
        \begin{equation}
            d(V_i) = \mathrm{dist}(v^1_i, v^2_i).
        \end{equation}
        By the compactness of $\overline{B}^{p+1}$ and $d(V_i) \rightarrow 0$, up to subsequence, we may assume that there is $v \in \overline{B}^{p+1}$, $v^1_i \rightarrow v$ and $v^2_i \rightarrow v$. By the closedness of $\overline{\mathbf{B}}^\F_{\lambda + 0.5, K}$ and $\partial \overline{\mathbf{B}}^\F_{\lambda + 1, K}$, one can easily verify that $v \in D_{small}(V) \cap \partial D_{large}(V)$, which contradicts to Claim 1. Thus, there exists a positive lower bound. $\hfill\blacksquare$\\

        Now, let's consider
        \begin{equation}
            P := \set{(V, v) \in \overline{\mathbf{B}}^\F_{\lambda, K} \times \overline{B}^{p+1} | v \neq m_{\bar{k}(K)}(V), v \in D_{small}(V)}.
        \end{equation}
        For any $(V, v) \in P$, $A^V_{\bar k(K)}(\phi^V_{\lambda, K}(v, t))$ is strictly decreasing with respect to $t > 0$ since $v \neq m_{\bar{k}(K)}(V)$. Therefore, we can define
        \begin{equation}
            \delta A_{\lambda, K}(V, v) := \lim_{t \rightarrow \infty} (A^V_{\bar k}(\phi^V_{\lambda, K}(v, 0)) - A^V_{\bar k}(\phi^V_{\lambda, K}(v, t))) > 0
        \end{equation}
        
        \textbf{Claim 3:} $\delta A_{\lambda, K}(V, v)$ has a uniform positive lower bound $2c(\lambda, K) > 0$ on $P$.

        Indeed, by the definition of $\phi^V_{\lambda, K}(\cdot, t)$, we know that,
        \begin{equation}
             \mathbf{F}(\phi^V_{\lambda, K}(v, t), \partial D_{large}(V)) \rightarrow 0,
        \end{equation}
        as $t \rightarrow \infty$. By Lemma \ref{lem:grad} and Claim $2$, we obtain for $\delta A_{\lambda, K}$ a uniform lower bound
        \begin{equation}
             2c(\lambda, K) := \frac{c_{0, \bar k(K)}}{2} \tilde d(\lambda, K)^2\,.
        \end{equation}
        $\hfill\blacksquare$\\

        Finally, to conclude, it suffices to find $T = T(\eta, K, \lambda) > 0$ such that for all $(V, v)\in P$ and $|v -m_{\bar{k}(K)}(V)| \geq \eta$,
        \begin{equation}
            A^V_{\bar k(K)}(\phi^V_{\lambda, K}(v, T)) < A^V_{\bar k(K)}(\phi^V_{\lambda, K}(v, 0)) - c(\lambda, K)\,.
        \end{equation}
        
        Following the proof of \cite[Lemma~4.5]{marques_morse_2016}, we argue by contradiction. Suppose there exists a sequence $(V_i, v_i) \in P$ and $|v_i - m_{\bar{k}(K)}(V_i)| \geq \eta$, such that
        \begin{equation}
            A^{V_i}_{\bar k(K)}(\phi^{V_i}_{\lambda, K}(v_i, t)) \geq A^{V_i}_{\bar k(K)}(\phi^{V_i}_{\lambda, K}(v_i, 0)) - c(\lambda, K).
        \end{equation}
        for any $t \in [0, i]$. By compactness, we may assume that $V_i \rightarrow V$ and $v_i \rightarrow v$ with $(V, v) \in P$ and $|v - m_{\bar{k}(K)}(V)| \geq \eta$. Hence, for any $t \geq 0$, 
        \begin{equation}
            A^V_{\bar k(K)}(\phi^V_{\lambda, K}(v, t)) \geq A^V_{\bar k(K)}(\phi^V_{\lambda, K}(v, 0)) - c(\lambda, K),
        \end{equation}
        contradicting to Claim 3.
    \end{proof}

    \begin{lemma}\label{lem:homotopy}
        Given a constant $\lambda > 0$, a nonempty finite subset $K \subset \mathbb{N}^+$ (assume that $\mathbf{B}^\F_{\lambda, K} \neq \emptyset$) and $c(\lambda, K)$ the constant in Lemma \ref{lem:mod}, for any $\delta \in (0, c(\lambda, K) / 4)$, any $l \in \set{0, 1, \cdots, p}$ and any continuous map,
        \begin{equation}
            \Theta: X^{l} \rightarrow \mathbf{B}^\F_{\lambda, K},
        \end{equation}
        where $X^l$ is an $l$-dimensional finite simplicial complex, there exists a homotopy map
        \begin{equation}
            H^{\Theta, \delta}_{\lambda,K} : X^l \times [0, 1] \rightarrow \mathbf{B}^\F_{\lambda+1, K},
        \end{equation}
        such that
        \begin{enumerate}
            \item $H^{\Theta, \delta}_{\lambda,K}(\cdot, 0) = \Theta(\cdot)$;
            \item $\forall x \in X^l, \|H^{\Theta, \delta}_{\lambda,K}(x, 1)\|(M) \leq \|\Theta(x)\|(M) - \tilde c(\lambda, K)$;
            \item $\forall x \in X^l, t\in [0,1], \|H^{\Theta, \delta}_{\lambda,K}(x, t)\|(M) - \|\Theta(x)\|(M) \leq \delta$.
        \end{enumerate}
        Here, $\tilde c(\lambda, K) = c(\lambda, K)/2$. More precisely,
        \begin{itemize}
            \item For $t \in [0, 1/2]$, $\F(H^{\Theta, \delta}_{\lambda,K}(x, t), \Theta(x)) \leq \delta$;
            \item For $t \in [1/2, 1]$, there exists $v(x) \in \overline{B}^{p+1}$, such that
                    \begin{equation}
                        H^{\Theta, \delta}_{\lambda, K}(x, t) = \phi^{\Theta(x)}_{\lambda, K}(v(x), (2t - 1)T)\,.
                    \end{equation}
        \end{itemize}
        
    \end{lemma}
    \begin{proof}
        Firstly, since $\F(\overline{\mathbf{B}}^\F_{\lambda, K}, \partial \overline{\mathbf{B}}^\F_{\lambda+0.5, K}) > 0$, there exists $\delta_0 \in (0, \delta)$ such that for any $V \in \overline{\mathbf{B}}^\F_{\lambda, K}$, if a varifold $W$ satisfies that $\F(V, W)\leq \delta_0$, then 
        \begin{equation}
            \begin{aligned}
                \|W\|(M) &\leq \|V\|(M) + \delta\,,\\
                W &\in \overline{\mathbf{B}}^\F_{\lambda+0.5, K}\,.
            \end{aligned}
        \end{equation}
        By the compactness of $\overline{\mathbf{B}}^\F_{\lambda,K}$ and $\overline{B}^{p+1}$, We can also choose $\delta_1 > 0$ such that 
        \begin{equation}
            \F((F^{\bar k(K)}_{v})_\#(V), V) \leq \delta_0 < \delta,
        \end{equation}
        for any $V \in \overline{\mathbf{B}}^\F_{\lambda,K}, v \in \overline{B}^{p+1}(\delta_1)$.
    
        Since $\Sigma_{\bar k(K)}$ is $(p + 1)$-unstable, and $\forall x \in X^l, \Theta(x) \in \overline{\mathbf{B}}^\F_{\lambda+0.5, K} \subset \overline{\mathbf{B}}^\F_{2\varepsilon_{\bar k(K)}}(\Sigma_{\bar{k}(K)})$, we can follow the construction of $\hat H^i$ in the proof of \cite[Theorem~5.1]{marques_morse_2016} and obtain a homotopy map
        \begin{equation}
            \hat H : X^l \times [0,1] \rightarrow \overline{B}^{p+1}(\delta_1),
        \end{equation}
        such that $\hat H(\cdot, 0) = 0$ and
        \begin{equation}
            \inf_{x\in X^l}|m_{\bar k(K)}(\Theta(x)) - \hat H(x,1)| \geq \eta > 0\,,
        \end{equation}
        for some $\eta = \eta(\delta, K, \Theta, \lambda) > 0$. 
        
        Finally, applying Lemma \ref{lem:mod} with $v = \hat H(x, 1)$ for $\Theta(x)$, we can obtain
        \begin{equation}
            H^{\Theta, \delta}_{\lambda, K}(x, t) = \begin{cases}
                (F^{\bar k}_{\hat H(x, 2t)})_{\#} \Theta(x) & t \in [0, 1/2]\\
                \phi^{\Theta(x)}_{\lambda, K}(\hat H(x, 1), (2t - 1)T) & t \in (1/2, 1]\,,
            \end{cases}
        \end{equation}
        which fulfills all the conditions.
    \end{proof}

    \begin{theorem}[Deformation Theorem A]\label{Thm:deform}
        Given a min-max sequence $\set{\Phi_i}_{i \in \mathbb{N}}$ for $\omega_p$, where $X_i = \mathrm{dmn}(\Phi_i)$ has dimension no greater than $p$, then up to a subsequence, there exists another sequence $\set{\Psi_i}$ such that
        \begin{enumerate}
            \item $\Psi_i$ is homotopic to $\Phi_i$ in the $\F$-topology;
            \item $\L(\set{\Psi_i}_{i\in \mathbb{N}}) = \omega_p$;
            \item There exists a positive function $\bar \varepsilon: \mathcal{U}(\omega_p) \rightarrow \mathbb{R}^+$ satisfying the following property. For any $\Sigma \in \mathcal{U}(\omega_p)$, there exists $i_0\in \mathbb{N}$ such that for all $i \geq i_0$,
            \begin{equation}
                |\Psi_i|(X_i) \cap \mathbf{B}^\F_{\bar \varepsilon(\Sigma)}(\Sigma) = \emptyset\,.
            \end{equation}
            Hence, $\mathbf{C}(\set{\Psi_i})\cap \mathcal{APR}_p\subset \mathcal{S}(\omega_p)$.
        \end{enumerate}
    \end{theorem}
    
    \begin{proof}
        The proof consists of 4 steps.\\
    
        \paragraph*{\textbf{Step 1.}} We construct a (finite or countable) sequence of $(p + 1)$-unstable minimal hypersurfaces $\set{\Sigma_k} \subset \mathcal{U}(\omega_p)$ and their associated quadruples $\set{(\varepsilon_k, c_{0, k}, \set{F^k_v}, m_k)}$ such that the following conditions hold.
        \begin{itemize}
            \item $\forall k_1 \neq k_2, \mathbf{B}^\F_{\varepsilon_{k_1}}(\Sigma_{k_1}) \not\subset \mathbf{B}^\F_{\varepsilon_{k_2}}(\Sigma_{k_2})$;
            \item $\mathcal{U}(\omega_p) \subset \bigcup^\infty_{k=1} \mathbf{B}^\F_{\varepsilon_{k}}(\Sigma_k)$ and moreover, $\forall m \in \mathbb{N}^+, \exists k_m \in \mathbb{N}^+$,
            \begin{equation}
                \mathcal{U}_m(\omega_p) \subset \bigcup^{k_m}_{i=1} \mathbf{B}^\F_{\varepsilon_{i}}(\Sigma_{i})\,;
            \end{equation}
            \item For any fixed $k$, $\bar K(\Sigma_k) := \set{k' | \mathbf{B}^\F_{2\varepsilon_{k'}}(\Sigma_{k'}) \cap \mathbf{B}^\F_{2\varepsilon_{k}}(\Sigma_k) \neq \emptyset}$ has
            \begin{equation}
                \#\bar K(\Sigma_k)<\infty\,.
            \end{equation}
            
        \end{itemize}
    
        For each $\Sigma \in \mathcal{U}(\omega_p)$, since it is $(p + 1)$-unstable, we can associate it with a quadruple $(\varepsilon_\Sigma, c_{0, \Sigma}, \set{F^\Sigma_v}, m_\Sigma)$ and then $\set{\mathbf{B}^\F_{\varepsilon_\Sigma}(\Sigma)}$ covers $\mathcal{U}(\omega_p)$. Note that we can always take a smaller $\varepsilon_\Sigma$ such that $\Sigma$ is still $(p + 1)$-unstable with the new quadruple.

        If $\mathcal{U}(\omega_p)$ itself is compact, we can find a finite cover $\set{\Sigma_k}$ with the quadruples above satisfying all the conditions.

        Otherwise, by Lemma \ref{lem:union_compact}, $\mathcal{U}(\omega_p) = \bigcup_m \mathcal{U}_m(\omega_p)$, where $\mathcal{U}_m(\omega_p)$ is compact. Let's construct the sequence inductively in $m$. For $m = 1$, for each $\Sigma \in \mathcal{U}_1(\omega_p)$, we take $\varepsilon_\Sigma$ to be smaller than $\frac{1}{2}$, and then take a finite cover $\set{\Sigma_k}^{k_1}_{k = 1}$ of $\mathcal{U}_1(\omega_p)$ with the smallest $k_1$. Moreover, we also have
        \begin{equation}
             \bigcup^{k_1}_{k=1} \mathbf{B}^\F_{2\varepsilon_k}(\Sigma_k) \subset \mathcal{U}_2(\omega_p)\,.
        \end{equation}
        Inductively, for $m > 1$, $\mathcal{U}_{m-1}(\omega_p)$ is covered by $\set{\Sigma_k}^{k_{m - 1}}_{k = 1}$, so we only need to consider the compact set $\mathcal{R}_m = \mathcal{U}_m(\omega_p)\backslash \bigcup^{k_{m - 1}}_{k = 1} \mathbf{B}^\F_{\varepsilon_k}(\Sigma_k)$. For each $\Sigma\in \mathcal{R}_m$, we take $\varepsilon_\Sigma$ to be smaller than $\frac{1}{m(m + 1)}$, and then take a finite cover $\set{\Sigma_k}^{k_m}_{k = 1 + k_{m - 1}}$ with the smallest $k_m$. Note that we still have
        \begin{equation}
             \bigcup^{k_m}_{k=1} \mathbf{B}^\F_{2\varepsilon_k}(\Sigma_k) \subset \mathcal{U}_{m + 1}(\omega_p)\,.
        \end{equation}
        It follows that this procedure generates a sequence $\set{\Sigma_k}_{k \in \mathbb{N}}$ fulfilling all the conditions.\\
    
        \paragraph*{\textbf{Step 2.}} We construct a function $\eta: \mathcal{U}(\omega_p) \rightarrow (0, 1]$, such that
        \begin{itemize}
            \item $\mathcal{N}_\eta := \bigcup_{\Sigma \in \mathcal{U}(\omega_p)} \mathbf{B}^\F_{2\eta(\Sigma)}(\Sigma) \subset \bigcup_{k} \mathbf{B}^\F_{\varepsilon_{k}}(\Sigma_k)$ and $\mathcal{N}^i_\eta := \bigcup_{\Sigma \in \mathcal{U}_m(\omega_p)} \mathbf{B}^\F_{2\eta(\Sigma)}(\Sigma) \subset \bigcup^{k_m}_{k = 1} \mathbf{B}^\F_{\varepsilon_{k}}(\Sigma_k)$.
            \item On each $\mathcal{U}_m(\omega_p)$, $\eta$ has a positive lower bound $\tilde \eta_m > 0$.
            \item For any $\lambda \in \set{0, 1, \cdots, p}, l \in \set{0, 1, \cdots, p} , K \subset \mathbb{N}, \delta \in (0, c(\lambda, K)/4)$ and any continuous map $\Theta: X^{l} \rightarrow \mathcal{N}_\eta \cap \mathbf{B}^\F_{\lambda, K} (\mathcal{N}_\eta \cap \mathbf{B}^\F_{\lambda, K} \neq \emptyset)$, there exists a homotopy map $H^{\Theta, \delta}_{\lambda, K}:~X^l \times [0, 1] \rightarrow \mathbf{B}^\F_{\lambda+1, K}$, such that
            \begin{enumerate}
                \item $H^{\Theta, \delta}_{\lambda, K}(\cdot, 0) = \Theta(\cdot)$.
                \item $\|H^{\Theta, \delta}_{\lambda, K}(x, t)\|(M) - \|\Theta(x)\|(M) \leq \delta, \forall t\in[0,1]$.
                \item For $t \in [0, 1/2]$, $\F(H^{\Theta, \delta}_{\lambda,K}(x, t), \Theta(x)) \leq \delta$.
                \item For $t \in [1/2, 1]$, there exists $v(x) \in \overline{B}^{p+1}$, such that
                    \begin{equation}
                        H^{\Theta, \delta}_{\lambda, K}(x, t) = \phi^{\Theta(x)}_{\lambda, K}(v(x), (2t - 1)T)\,.
                    \end{equation}
                \item $H^{\Theta, \delta}_{\lambda, K}(\cdot, 1) \not\in \mathcal{N}_{\eta}$.
            \end{enumerate}
        \end{itemize}
        
        \textbf{Claim 1.} There exists a \textit{continuous} positive function $\eta_1: \mathcal{U}(\omega_p)\rightarrow \mathbb{R}^+$ such that for any $i \in \mathbb{N}, \Sigma \in \mathcal{U}_m(\omega_p)$, 
        \begin{equation}
            \mathbf{B}^\F_{2\eta_1(\Sigma)}(\Sigma) \subset \bigcup^{k_m}_{k = 1} \mathbf{B}^\F_{\varepsilon_{k}}(\Sigma_k)\,.
        \end{equation}
    
        We can define $\eta_1$ inductively. 
        If $\Sigma \in \mathcal{U}_1(\omega_p)$, then 
            \begin{equation}
                \eta_1(\Sigma):=\max\set{r > 0 | \mathbf{B}^\F_{2r}(\Sigma) \subset \bigcup^{k_1}_{k = 1} \mathbf{B}^\F_{\varepsilon_{k}}(\Sigma_k)}\,.
            \end{equation}
        If $m \geq 2$ is the least number such that $\Sigma \in \mathcal{U}_m(\omega_p)\backslash \mathcal{U}_{m-1}(\omega_p)$, we define 
            \begin{equation}
                    \eta_1(\Sigma) := \min\left\{\begin{array}{ll}\max\left\{r > 0 | \mathbf{B}^\F_{2r}(\Sigma) \subset \bigcup^{k_m}_{k = 1} \mathbf{B}^\F_{\varepsilon_{k}}(\Sigma_k)\right\},\\
                    \min\left\{\eta_1(\Sigma') + \F(\Sigma, \Sigma')| \Sigma' \in \mathcal{U}_{m-1}(\omega_p)\right\}\end{array}\right\} > 0\,.
            \end{equation}
        Note that $\eta_1$ is continuous and thus has a positive lower bound on each compact set $\mathcal{U}_m(\omega_p)$. $\hfill\blacksquare$
        
        We define $\eta(\Sigma)$ to be the largest number in $(0, \min(\eta_1(\Sigma), 1)]$ such that 
        \begin{equation}
            \forall V \in \mathbf{B}^\F_{\eta(\Sigma)}(\Sigma),\lambda \in \set{0, 1, \cdots, p}, |\omega_p - \|V\|(M)| \leq \tilde{c}(\lambda, \tilde{K})/3, 
        \end{equation} 
        for any nonempty subset $\tilde K \subset \set{k|\mathbf{B}^\F_{\eta_1(\Sigma)}(\Sigma) \cap \mathbf{B}^\F_{2\varepsilon_k}(\Sigma_k) \neq \emptyset}$.
    
        To show that $\eta$ has a positive lower bound $\tilde \eta_m$ on each $\mathcal{U}_m(\omega_p)$, we observe that for any $\Sigma \in \mathcal{U}_m(\omega_p)$, $\set{k|\mathbf{B}^\F_{\eta_1(\Sigma)}(\Sigma) \cap \mathbf{B}^\F_{2\varepsilon_k}(\Sigma_k) \neq \emptyset}$ is a subset of 
        \begin{equation}
             K_m = \bigcup^{k_m}_{k = 1}\bar K(\Sigma_k)\,,
        \end{equation}
        since $\bigcup_{\Sigma\in \mathcal{U}_m(\omega_p)} \mathbf{B}^\F_{2\eta_1(\Sigma)}(\Sigma) \subset \bigcup^{k_m}_{k = 1} \mathbf{B}^\F_{\varepsilon_{k}}(\Sigma_k)$.
        By the third bullet in \textbf{Step 1}, $K_m$ is also a finite set. It follows from the compactness of $\mathcal{U}_m(\omega_p)$ that there exists a positive number $\tilde \eta_m$, such that for any $\Sigma \in \mathcal{U}_m(\omega_p)$ and any $V \in \mathbf{B}^\F_{\tilde \eta_m}(\Sigma)$,
        \begin{equation}
            |\omega_p - \|V\|(M)| \leq \min_{K \subset K_m, \lambda \in \set{0, 1, 2, \cdots, p}}\tilde{c}(\lambda, K)/3,
        \end{equation}
        and thus, $\eta(\Sigma) \geq \tilde \eta_m > 0$ on $\mathcal{U}_m(\omega_p)$.

        For the third bullet, we use Lemma \ref{lem:homotopy} to construct the homotopy map $H^{\Theta, \delta}_{\lambda,K}$. It follows immediately that conditions (1), (2), (3) and (4) are satisfied. Suppose that condition (5) doesn't hold, and then there exist some $x \in X^l$ and some $\Sigma \in \mathcal{U}(\omega_p)$, such that $H^{\Theta, \delta}_{\lambda,K}(x, 1) \in \mathbf{B}^\F_{\eta(\Sigma)}(\Sigma)\cap \mathbf{B}^\F_{\lambda+1, K}$. Hence,  by the definition of $\eta$, we have
        \begin{equation}
             \omega_p - \|H^{\Theta, \delta}_{\lambda,K}(x, 1)\|(M) < \tilde{c}(\lambda,K)/2\,.
        \end{equation}
        On the other hand, we also have $H^{\Theta, \delta}_{\lambda, K}(x, 0) = \Theta(x) \in \mathbf{B}^\F_{\eta(\Sigma)}(\Sigma)\cap \mathbf{B}^\F_{\lambda+1, K}$, and thus,
        \begin{equation}
            \|\Theta(x)\|(M) - \omega_p < \tilde{c}(\lambda, K)/2\,,
        \end{equation}
        so $\|\Theta(x)\|(M) - \|H^{\Theta, \delta}_{\lambda, K}(x, 1)\|(M) < \tilde{c}(\lambda, K)$, which gives a contradiction to (2) of Lemma \ref{lem:homotopy}.\\
    
        \paragraph*{\textbf{Step 3.}} We show that the homotopy map defined in \textbf{Step 2} would not push a varifold too close to $\mathcal{U}(\omega_p)$, provided that it is far away from $\mathcal{U}(\omega_p)$ initially.

        More precisely, we would like to show that for any $\Sigma \in \mathcal{U}(\omega_p)$, there exist two sequences $\set{\varepsilon_q(\Sigma)}^{p}_{q = 1} \subset \mathbb{R}^+$ and $\set{a_q(\Sigma)}^{p+1}_{q=1}$ with $a_1(\Sigma) = 2$ and $a_q(\Sigma) \geq 2^q$ satisfying the following property. In the third bullet of \textbf{Step 2}, for any $q \in \set{1,\dots, p}$, $\Sigma \in \mathcal{U}(\omega_p)$ with $\mathbf{B}^\F_{\lambda + 1, K}\cap \mathbf{B}^\F_{\eta(\Sigma)}(\Sigma) \neq \emptyset$, and $x \in X^l$ with $\Theta(x) \notin \mathbf{B}^\F_{\eta(\Sigma)/ a_q(\Sigma)}(\Sigma)$ as well as $\|\Theta(x)\|(M) - \omega_p \leq \varepsilon_q(\Sigma)$, we have for all $t \geq 0$ and $\delta \in (0, \min(\eta(\Sigma)/(4a_q(\Sigma)), \varepsilon_q(\Sigma)))$,
        \begin{equation}
            H^{\Theta, \delta}_{\lambda,K}(x, t) \notin \mathbf{B}^\F_{\eta(\Sigma)/a_{q+1}(\Sigma)}(\Sigma)\,.
        \end{equation}
        Moreover, in each $\mathcal{U}_m(\omega_p)$, $\varepsilon_q(\Sigma)$ has a positive lower bounds $\tilde \varepsilon^m_q$ and $a_q(\Sigma)$ has a uniform upper bounds $\tilde a^m_q$.
    
        Let's construct $a_q(\Sigma)$ and $\varepsilon_q(\Sigma)$ inductively. Note that a priori, we have $a_1(\Sigma) = 2$.
        
        For $\Sigma \in \mathcal{U}(\omega_p)$, suppose that we have defined $a_q(\Sigma)$ and, if $q > 1$, $\varepsilon_{q-1}(\Sigma)$. For any natural number $k$ with $\overline{\mathbf{B}}^\F_{p + 1, \set{k}}\cap \mathbf{B}^\F_{\eta(\Sigma)}(\Sigma) \neq \emptyset$, and $V\in \overline{\mathbf{B}}^\F_{p + 1, \set{k}}$, if $\left((F^{k}_{\cdot})_{\#}(V)\right)^{-1}\overline{\mathbf{B}}^\F_{\eta(\Sigma)/(2a_{q}(\Sigma))}(\Sigma) \neq \emptyset$ and $\left((F^{k}_{\cdot})_{\#}(V)\right)^{-1}\partial\mathbf{B}^\F_{3\eta(\Sigma)/(4a_{q}(\Sigma))}(\Sigma)\neq \emptyset$, we define
        \begin{equation}
            \begin{aligned}
                d_{q, \Sigma, k}(V) := \mathrm{dist}&\left( \left((F^{k}_{\cdot})_{\#}(V)\right)^{-1}\partial \mathbf{B}^\F_{3\eta(\Sigma)/(4a_q(\Sigma))}(\Sigma),\right. \\ &\;\;\left.\left((F^{k}_{\cdot})_{\#}(V)\right)^{-1}\overline{\mathbf{B}}^\F_{\eta(\Sigma)/(2a_{q}(\Sigma))}(\Sigma)\right)\,.
            \end{aligned}
        \end{equation}
        Otherwise, $d_{q, \Sigma, k}(V) := +\infty$. By Lemma \ref{lem:dist_F}, we have
        \begin{equation}
            d_{q, \Sigma, k}(V) \geq h_{\Sigma_k}(\eta(\Sigma)/(4a_q(\Sigma))).
        \end{equation}
        Let's fix a $k'$ such that $\overline{\mathbf{B}}^\F_{p + 1, \set{k'}}\cap \mathbf{B}^\F_{\eta(\Sigma)}(\Sigma) \neq \emptyset$ and observe any $k$ with $\overline{\mathbf{B}}^\F_{p + 1, \set{k}}\cap \mathbf{B}^\F_{\eta(\Sigma)}(\Sigma) \neq \emptyset$ is inside $\bar K(\Sigma_{k'})$ in the third bullet of \textbf{Step 1} and $\#\bar K(\Sigma_{k'}) < \infty$. Hence, we can define
        \begin{equation}
            \varepsilon_q(\Sigma):= \min_{\substack{k, \\ \overline{\mathbf{B}}^\F_{p + 1, \set{k}}\cap \mathbf{B}^\F_{\eta(\Sigma)}(\Sigma) \neq \emptyset}} \frac{c_{0, k}}{16} \left(h_{\Sigma_k}(\eta(\Sigma)/(4a_q(\Sigma)))\right)^2\,.
        \end{equation}
        
        \textbf{Claim 2.} There exists $\tilde \varepsilon^m_q > 0$, such that $\inf_{\Sigma \in \mathcal{U}_m(\omega_p)}\varepsilon_q(\Sigma)\geq \tilde\varepsilon^m_q$.

        For any fixed $\mathcal{U}_m(\omega_p)$, by the first bullet of \textbf{Step 2}, we know that
        \begin{equation}
            K'_m = \set{k| \exists \Sigma \in \mathcal{U}_m(\omega_p), \overline{\mathbf{B}}^\F_{p + 1, \set{k}} \cap \mathbf{B}^\F_{\eta(\Sigma)}(\Sigma) \neq \emptyset} \subset \bigcup^{k_m}_{k = 1} \bar K(\Sigma_k)\,,
        \end{equation}
        which implies that $\# K'_m < \infty$.

        By our assumption and the definition of $\eta(\Sigma)$, $\eta(\Sigma)/a(\Sigma)$ has a positive lower bound, say $c'_m > 0$ among all $\Sigma \in \mathcal{U}_m(\omega_p)$. Hence,
        \begin{equation}
            \begin{aligned}
                \inf_{\Sigma \in \mathcal{U}_m(\omega_p)} \varepsilon_q(\Sigma) &= \inf_{\Sigma \in \mathcal{U}_m(\omega_p)}\min_{\substack{k, \\ \overline{\mathbf{B}}^\F_{p + 1, \set{k}} \cap \mathbf{B}^\F_{\eta(\Sigma)}(\Sigma) \neq \emptyset}} \frac{c_{0, k}}{16} \left(h_{\Sigma_k}(\eta(\Sigma)/(4a_q(\Sigma)))\right)^2\\
                &= \min_{k \in K'_m} \inf_{\substack{\Sigma \in \mathcal{U}_m(\omega_p), \\ \overline{\mathbf{B}}^\F_{p + 1, \set{k}} \cap \mathbf{B}^\F_{\eta(\Sigma)}(\Sigma) \neq \emptyset}} \frac{c_{0, k}}{16} \left(h_{\Sigma_k}(\eta(\Sigma)/(4a_q(\Sigma)))\right)^2\\
                &\geq \min_{k \in K'_m} \frac{c_{0, k}}{16} h_{\Sigma_k}(c'_m)^2 > 0\,.
            \end{aligned}
        \end{equation}
        $\hfill\blacksquare$

        Let's take $a_{q + 1}(\Sigma) \geq 2 a_q(\Sigma) + 1$ to be the smallest number such that for any $V \in \overline{\mathbf{B}}^\F_{\eta(\Sigma)/a_{q + 1}(\Sigma)}(\Sigma)$, 
        \begin{equation}
            \|V\|(M) \geq \omega_p - \varepsilon_q(\Sigma)\,.
        \end{equation}

        \textbf{Claim 3.} There exists $\tilde a^m_{q + 1} > 0$, such that $\inf_{\Sigma \in \mathcal{U}_m(\omega_p)}a_{q + 1}(\Sigma)\geq \tilde a^m_{q + 1}$.

        Since $\eta(\Sigma) \leq 1$, it suffices to there exists a constant $c'_m > 0$ such that for any $\Sigma \in \mathcal{U}_m(\omega_p)$ and $V \in \overline{\mathbf{B}}^\F_{c'_m}(\Sigma)$, 
        \begin{equation}
            \|V\|(M) \geq \omega_p - \tilde \varepsilon^m_q\,.
        \end{equation}
        Then, we have $a_{q + 1}(\Sigma) \leq 1/ c'_m$ for any $\Sigma \in \mathcal{U}_m(\omega_p)$.

        Let's argue by contradiction. Suppose not and there exists a sequence $\set{\Sigma_i} \in \mathcal{U}_m(\omega_p)$ and $V_i \in \overline{\mathbf{B}}^\F_{1/i}(\Sigma)$ such that
        \begin{equation}
            \|V_i\|(M) < \omega_p - \tilde \varepsilon^m_q\,.
        \end{equation}
        By the compactness of $\mathcal{U}_m(\omega_p)$, up to a subsequence, we may assume that $\Sigma_i \rightarrow \Sigma$ and $V_i \rightarrow \Sigma$. However, $\|\Sigma\|(M) = \omega(p)$ and thus
        \begin{equation}
            \omega_p < \omega_p - \tilde \varepsilon^m_q\,,
        \end{equation}
        giving a contradiction. $\hfill \blacksquare$

        \textbf{Claim 4.} $H^{\Theta, \delta}_{\lambda,K}(x, t) \notin \mathbf{B}^\F_{\eta(\Sigma)/a_{q+1}(\Sigma)}(\Sigma)$ provided that $\Theta(x) \notin \mathbf{B}^\F_{\eta(\Sigma)/ a_q(\Sigma)}(\Sigma)$ and $\|\Theta(x)\|(M) - \omega_p \leq \varepsilon_q(\Sigma)$.

        Indeed, since $\delta < \eta(\Sigma) / (4a_q(\Sigma))$ and $\Theta(x) \notin \mathbf{B}^\F_{\eta(\Sigma)/a_q}(\Sigma)$, for $t\in[0, 1/2]$, we have $\F(H^{\Theta,\delta}_{\lambda, K}(x,t), \Sigma) \geq 3\eta(\Sigma)/(4\cdot a_k) \geq \eta(\Sigma)/a_{q+1}(\Sigma)$.

        To show that $\F(H^{\Theta, \delta}_{\lambda, K}(x,t), \Sigma) \geq \eta(\Sigma)/a_{q+1}(\Sigma)$ for $t \in (1/2, 1]$, we first note that $\delta < \varepsilon_q(\Sigma)$ and thus, $\|H^{\Theta,\delta}_{\lambda, K}(x,1/2)\|(M) \leq \omega_p + 2 \varepsilon_q(\Sigma)$. Let's argue by contradiction and suppose there exists a $t_1 \in (1/2, 1]$ such that $H^{\Theta, \delta}_{\lambda, K}(x,t_1) \in \overline{\mathbf{B}}^\F_{\eta(\Sigma)/a_{q+1}(\Sigma)}(\Sigma)$. In this case, let $k = \bar k(K)$ and
        \begin{equation}
            \begin{aligned}
                &\mathrm{dist}\left(\left((F^{k}_{\cdot})_{\#}(\Theta(x))\right)^{-1}(H^{\Theta, \delta}_{\lambda, K}(x,t_1)), \left((F^{k}_{\cdot})_{\#}(\Theta(x))\right)^{-1}(H^{\Theta, \delta}_{\lambda, K}(x,1/2))\right)\\
                \geq & \mathrm{dist}\left(\left((F^{k}_{\cdot})_{\#}(\Theta(x))\right)^{-1}(\overline{\mathbf{B}}^\F_{\eta(\Sigma)/a_{q+1}(\Sigma)}(\Sigma)), \left((F^{k}_{\cdot})_{\#}(\Theta(x))\right)^{-1}(\partial{\mathbf{B}^\F_{3\eta(\Sigma)/(4a_{q})(\Sigma)}(\Sigma)})\right)\\
                \geq & d_{q, \Sigma, k}(\Theta(x)) \geq \tilde d_{q, \Sigma, k} > 0\,.
            \end{aligned}
        \end{equation}
        It follows from Lemma \ref{lem:grad} that
        \begin{equation}
            \begin{aligned}
                \|H^{\Theta, \delta}_{\lambda, K}(x,t_1)\|(M) &\leq \|H^{\Theta, \delta}_{\lambda, K}(x,1/2)\|(M) -\frac{c_{0, k}}{2}(\tilde d_{q, \Sigma, k})^2 \\
                    &\leq \omega_p + 2\varepsilon_q(\Sigma) - 8 \varepsilon_q(\Sigma)\\
                    &= \omega_p - 6 \varepsilon_q(\Sigma)\,,
            \end{aligned}
        \end{equation}
        contradicting to the definition of $a_{q + 1}(\Sigma)$. $\hfill\blacksquare$
    
        Before the end of this step, we define 
        \begin{equation}
            \varepsilon_m := \min_{q \in \set{1, 2, \cdots, p}}\inf_{\Sigma \in \mathcal{U}_m(\omega_p)}\min(\eta(\Sigma)/(4a_q(\Sigma)), \varepsilon_q(\Sigma)) > 0\,.
        \end{equation}\\

        \paragraph*{\textbf{Step 4 [Hierarchical Deformations]}} Up to a subsequence, we construct the homotopy map $H_i: X \times [0, 1] \rightarrow \Z_n(M; \F; \Zn_2)$, such that $H_i(\cdot, 0) = \Phi_i$ and $H_i(\cdot, 1) = \Psi_i$ with all the required properties.
    
        We choose a subsequence $\set{\Phi_{j_i}}$ such that $\sup\set{\M(\Phi_{j_i}(x))} \leq \omega_p + \varepsilon_i / 2$ where $\varepsilon_i$ is defined in the end of \textbf{Step 3}. For simplicity, we still denote the sequence by $\set{\Phi_i}$.
    
        For a fixed $\Phi_i: X_i \rightarrow \Z_n(M; \F; \Zn_2)$ and a positive function $\psi: \mathcal{U}_i(\omega_p) \rightarrow \mathbb{R}^+$, we define 
        \begin{equation}
            \mathcal{N}^i_{\psi} := \bigcup_{\Sigma \in \mathcal{U}_i(\omega_p)} \mathbf{B}^\F_{2\psi(\Sigma)}(\Sigma)\,,
        \end{equation}
        and then it suffices to deform $\Phi_i$ to $\Psi_i$ such that $|\Psi_i| \cap \mathcal{N}^i_{\eta/(2a_{p+1})} = \emptyset$.
    
        To do so, we start with a finer subdivision on $X_i$ such that for each (closed) face $F$ of $X_i$,
        \begin{enumerate}
            \item If $|\Phi_i|(F) \cap \mathcal{N}^i_{\eta/2} \neq \emptyset$, then $|\Phi_i|(F) \subset \mathcal{N}^i_{\eta}$.  This is possible, since $\eta$ has a uniform lower  bound on $\mathcal{U}_i(\omega_p)$, which implies that $\F(\partial \mathcal{N}^i_{\eta}, \mathcal{N}^i_{\eta / 2}) > 0$. The set consisting of all such faces will be denoted by $\mathcal{B}$.
            \item For $F \in \mathcal{B}$, we can assign a nonempty finite set $K(F)\subset \mathbb{N}$, such that $|\Phi_i|(F)\subset \mathbf{B}^\F_{K(F)}$ and if $F \subset F'$both in $\mathcal{B}$, we have $K(F)\supset K(F')$. This can be defined inductively on the decreasing dimensions as long as the subdivision is fine enough.
        \end{enumerate}
    
        Now, we would like to construct the homotopy map $H_i$ inductively on the $k$-skeleton of $X_i$, i.e., $X^{(k)}_i$. 
        
        For $k = 0$, we apply \textbf{Step 2} with $\lambda = 0$ and $\delta < \varepsilon_i/4$ to all the $0$-cells in $\mathcal{B}$. For all the $0$-cells outside $\mathcal{B}$, we simply construct a constant homotopy map. Thus, we obtain a map $H^{(0)}_i$ defined on $X^{(0)}_i \times [0, 1]$. Moreover, $|H^{(0)}_i|(X^{(0)}_i \times 1) \cap \mathcal{N}^i_{\eta / a_1} = \emptyset$ and $|H^{(0)}_i|(x, [0, 1]) \subset \mathbf{B}^\F_{1, K(x)}$ for any $x \in \mathcal{B}^{(0)}$.

        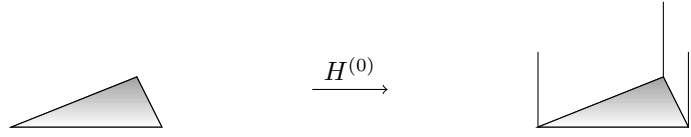
\begin{figure}[!h]
            \centering
            \begin{tikzpicture}[]
                \coordinate (A1) at (1,0,0);
                \coordinate (A2) at (-1,0,0);
                \coordinate (A3) at (0,0,-{sqrt(3)});

                \draw (A1) -- (A2) -- (A3) -- cycle;
                \draw [fill opacity=0.7, shade] (A2) -- (A3) -- (A1) -- cycle;

                \coordinate (a1) at (3, 0.5, 0);
                \coordinate (a2) at (4, 0.5, 0);
                \draw[->] (a1) -- node [midway, above] {$H^{(0)}$} (a2);

                \coordinate (A1') at (8,0,0);
                \coordinate (A2') at (6,0,0);
                \coordinate (A3') at (7,0,-{sqrt(3)});
                \coordinate (B1') at (8,1,0);
                \coordinate (B2') at (6,1,0);
                \coordinate (B3') at (7,1,-{sqrt(3)});

                \draw (A1') -- (B1');
                \draw (A2') -- (B2');
                \draw (A3') -- (B3');
                \draw (A1') -- (A2') -- (A3') -- cycle;
                \draw [fill opacity=0.7, shade] (A2') -- (A3') -- (A1') -- cycle;
            \end{tikzpicture}
            \caption{Deformation $H^{(0)}$}
        \end{figure}
            
        Inductively, suppose that we have constructed $H^{(k-1)}_i$ on $X^{(k-1)}_i$ ($k \geq 1$), and now we consider the $k$-cells in $X_i$. Fixing $F_k \in X^{(k)}_i\backslash X^{(k-1)}_i$, then $\partial F_k \in X^{(k-1)}_i$. Note that $F'_k := F_k \cup \left(\partial F_k \times [0, 1]\right)$ is homeomorphic to $F_k$ ($\cong D^k$), so we can define $\Theta$ on $F'_k$ by concatenating $H^{(k-1)}_i$ on $\partial F_k \times [0, 1]$ and $\Phi_i$ on $F_k$. Now, we would like to construct a homotopy map $\tilde H^{(k)}_i$ with initial data $\Theta$.
        
        If $F_k \notin \mathcal{B}$, then according to the definition of $\mathcal{B}$, no cell in $\partial F_k$ belongs to $\mathcal{B}$ either. Thus, the homotopy map $\tilde H^{(k)}_i$ on $F'_k$ in this case can be defined as a constant homotopy. It's worthy to note that 
        \begin{equation}
            |\tilde H^{(k)}_i(F'_k \times 1 \cup \partial F'_k \times [0, 1])| \cap \mathcal{N}^i_{\eta / a_{k + 1}} = \emptyset\,.
        \end{equation}

        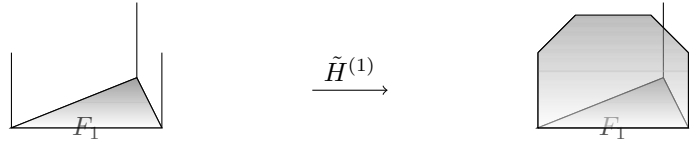
\begin{figure}[!h]
            \centering
            \begin{tikzpicture}[]
                \coordinate (A1) at (1,0,0);
                \coordinate (A2) at (-1,0,0);
                \coordinate (A3) at (0,0,-{sqrt(3)});
                \coordinate (B1) at (1,1,0);
                \coordinate (B2) at (-1,1,0);
                \coordinate (B3) at (0,1,-{sqrt(3)});
                \draw (A1) -- (B1);
                \draw (A2) -- (B2);
                \draw (A3) -- (B3);
                \draw (A1) -- (A2) -- (A3) -- cycle;
                \draw [fill opacity=0.7, shade] (A2) -- (A3) -- (A1) -- node [midway] {$F_1$} cycle;

                \coordinate (a1) at (3, 0.5, 0);
                \coordinate (a2) at (4, 0.5, 0);
                \draw[->] (a1) -- node [midway, above] {$\tilde H^{(1)}$} (a2);

                \coordinate (A1') at (8,0,0);
                \coordinate (A2') at (6,0,0);
                \coordinate (A3') at (7,0,-{sqrt(3)});
                \coordinate (B1') at (8,1,0);
                \coordinate (B2') at (6,1,0);
                \coordinate (B3') at (7,1,-{sqrt(3)});
                \coordinate (C1') at (7.5, 1.5, 0);
                \coordinate (C2') at (6.5, 1.5, 0);

                \draw (A1') -- (B1');
                \draw (A2') -- (B2');
                \draw (A3') -- (B3');
                \draw (B1') -- (C1');
                \draw (B2') -- (C2');
                \draw (A1') -- (A2') -- (A3') -- cycle;
                \draw [fill opacity=0.7, shade] (A2') -- (A3') -- (A1') -- node [midway] {$F_1$} cycle;
                \draw [fill opacity=0.4, shade] (A1') -- (A2') -- (B2') -- (C2') -- (C1') -- (B1') -- cycle;
            \end{tikzpicture}
            \caption{Deformation $\tilde H^{(1)}$}
        \end{figure}

        If $F_k \in \mathcal{B}$, similar to the refinement mentioned at the beginning, by the second condition on the subdivision, we know that $|\Theta|(F'_k) \subset \mathbf{B}^\F_{k, K(F_k)}$. We can make a subdivision on $F'_k$ such that, each (closed) $k$-dimensional face $\tilde F$ with $|\Theta|(\tilde F)\cap \mathcal{N}^i_{\eta/ a_k} \neq \emptyset$ must have $|\Theta|(\tilde F) \subset \mathcal{N}^i_{\eta}$. The union of all such $k$-dimensional faces now will be denoted by $Y$. Note that $|\Theta|(\partial Y) \cap \mathcal{N}^i_{\eta / a_k} = \emptyset$, since $|\Theta|(\partial F'_k) \cap \mathcal{N}^i_{\eta / a_k} = \emptyset$ by induction. With this new subdivision, we can define a map 
        \begin{equation}
            \hat H: F'_k \times [0, 1] \cup Y \times [1,2] \rightarrow \Z_n(M; \F; \Zn_2) 
        \end{equation}
        such that $\hat H(\cdot, t) := \Theta(\cdot)$ for $t \in [0,1]$, and $\hat H(\cdot, t+1) = H^{|\Theta|, \delta}_{k, K(F_k)}(\cdot, t)$ for $t \in [0,1]$ in \textbf{Step 2} with $\delta < \varepsilon_i$. By the construction of $Y$, we have
        \begin{equation}
            |\hat H((F'_k - Y) \times 1 \cup \partial F'_k \times [0, 1])| \cap \mathcal{N}^i_{\eta/a_{k+1}} = \emptyset\,.
        \end{equation}
        By (5) in the third bullet of \textbf{Step 2},
        \begin{equation}
            |\hat H(Y \times 2)| \cap \mathcal{N}^i_{\eta/a_{k+1}} = \emptyset\,.
        \end{equation}
        By \textbf{Step 3},
        \begin{equation}
             |\hat H(\partial Y \times [1, 2])| \cap \mathcal{N}^i_{\eta/a_{k+1}} = \emptyset\,.
        \end{equation} 
        We can derive $\tilde H^{(k)}_i: F'_k \times [0, 1] \rightarrow \Z_n(M; \F; \Zn_2)$ from $\hat H$ induced by the homeomorphism 
        \begin{equation}
            \begin{aligned}
                (F'_k \times [0,1], &\partial F'_k \times [0,1], F'_k \times 1) \cong 
                 (F'_k \times [0, 1] \cup Y \times [1,2], \\
                 & \partial F'_k \times [0,1], (F'_k\backslash Y)\times 1) \cup \left(Y \times 2\right) \cup \left(\partial Y \times [1,2]\right))\,,
            \end{aligned}
        \end{equation}
        By the property of $\hat H$, we also have
        \begin{equation}
            |\tilde H^{(k)}_i(F'_k \times 1 \cup \partial F'_k \times [0, 1])| \cap \mathcal{N}^i_{\eta / a_{k + 1}} = \emptyset\,.
        \end{equation}

        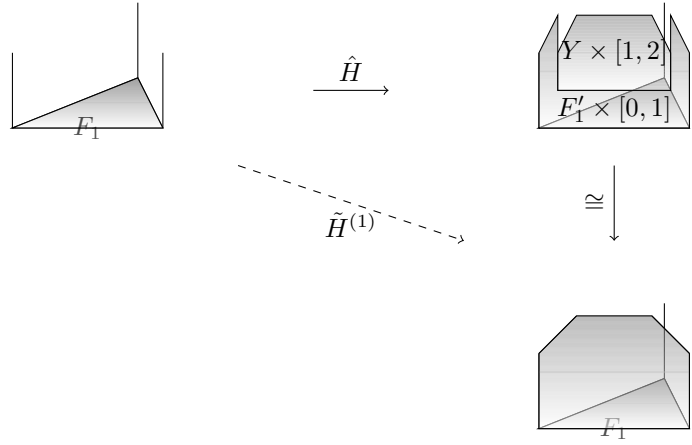
\begin{figure}[!h]
            \centering
            \begin{tikzpicture}[]
                \coordinate (A1) at (1,0,0);
                \coordinate (A2) at (-1,0,0);
                \coordinate (A3) at (0,0,-{sqrt(3)});
                \coordinate (B1) at (1,1,0);
                \coordinate (B2) at (-1,1,0);
                \coordinate (B3) at (0,1,-{sqrt(3)});
                \draw (A1) -- (B1);
                \draw (A2) -- (B2);
                \draw (A3) -- (B3);
                \draw (A1) -- (A2) -- (A3) -- cycle;
                \draw [fill opacity=0.7, shade] (A2) -- (A3) -- (A1) -- node [midway] {$F_1$} cycle;

                \coordinate (a1) at (3, 0.5, 0);
                \coordinate (a2) at (4, 0.5, 0);
                \draw[->] (a1) -- node [midway, above] {$\hat H$} (a2);

                \coordinate (A1') at (8,0,0);
                \coordinate (A2') at (6,0,0);
                \coordinate (A3') at (7,0,-{sqrt(3)});
                \coordinate (B1') at (8,1,0);
                \coordinate (B2') at (6,1,0);
                \coordinate (B3') at (7,1,-{sqrt(3)});
                \coordinate (C1') at (7.75, 1.5, 0);
                \coordinate (C2') at (6.25, 1.5, 0);
                \coordinate (C3') at (7.75, 0.5, 0);
                \coordinate (C4') at (6.25, 0.5, 0);
                \coordinate (C5') at (7.75, 1, 0);
                \coordinate (C6') at (6.25, 1, 0);
                \coordinate (C7') at (7.5, 1.5, 0);
                \coordinate (C8') at (6.5, 1.5, 0);

                \draw (A1') -- (B1');
                \draw (A2') -- (B2');
                \draw (A3') -- (B3');
                \draw (A1') -- (A2') -- (A3') -- cycle;
                \draw [fill opacity=0.7, shade] (A2') -- (A3') -- (A1') -- cycle;
                \draw [fill opacity=0.4, shade] (A1') -- (A2') -- (B2') -- (C2') -- (C4') -- (C3') -- (C1') -- (B1') -- cycle;
                \node at (7, 0.25, 0) {$F'_1\times [0,1]$}; 
                \draw [fill opacity=0.4, shade] (C5') -- (C3') -- (C4') -- (C6') -- (C8') -- (C7') -- cycle;
                \node at (7, 1, 0) {$Y\times [1,2]$}; 
                
                \coordinate (a3) at (7, -0.5, 0);
                \coordinate (a4) at (7, -1.5, 0);
                \draw[->] (a3) -- node [midway, left] {$\cong$} (a4);

                \coordinate (a5) at (2, -0.5, 0);
                \coordinate (a6) at (5, -1.5, 0);
                \draw[->, dashed] (a5) -- node [midway, left,below] {$\tilde H^{(1)}$} (a6);
                
                \coordinate (A1'') at (8,-4,0);
                \coordinate (A2'') at (6,-4,0);
                \coordinate (A3'') at (7,-4,-{sqrt(3)});
                \coordinate (B1'') at (8,-3,0);
                \coordinate (B2'') at (6,-3,0);
                \coordinate (B3'') at (7,-3,-{sqrt(3)});
                \coordinate (C1'') at (7.5, -2.5, 0);
                \coordinate (C2'') at (6.5, -2.5, 0);

                \draw (A1'') -- (B1'');
                \draw (A2'') -- (B2'');
                \draw (A3'') -- (B3'');
                \draw (B1'') -- (C1'');
                \draw (B2'') -- (C2'');
                \draw (A1'') -- (A2'') -- (A3'') -- cycle;
                \draw [fill opacity=0.7, shade] (A2'') -- (A3'') -- (A1'') -- node [midway] {$F_1$} cycle;
                \draw [fill opacity=0.4, shade] (A1'') -- (A2'') -- (B2'') -- (C2'') -- (C1'') -- (B1'') -- cycle;
            \end{tikzpicture}
            \caption{Deformations $\hat H$ and $\tilde H^{(1)}$}
        \end{figure}
        
        For both cases, we can derive $H^{(k)}_i: F_k \times [0, 1] \rightarrow \Z_n(M; \F; \Zn_2)$ from $\tilde{H}^{(k)}_i$ induced by the homeomorphism 
        \begin{equation}
            (F_k \times [0,1], F_k \times 1) \cong (F'_k \times [0,1], F'_k \times 1 \cup \partial F'_k \times [0,1]),
        \end{equation}
        satisfying that $H^{(k)}_i|_{\partial F_k \times [0,1]} = H^{(k-1)}_i|_{\partial F_k \times [0,1]}$ and $H^{(k)}_i|_{F_k}(\cdot, 0) = \Phi_i|_{F_k}$. Note that,
        \begin{equation}
            |H^{(k)}_i(F_k \times 1)| = |\tilde H^{(k)}_i(F'_k \times 1 \cup \partial F'_k \times [0, 1])|\,,    
        \end{equation}
        Therefore, we can concatenate all the $H^{(k)}_i$'s defined on $F_k \times [0, 1]$'s and obtain
        \begin{equation}
            H^{(k)}_i: X^{(k)}_i \rightarrow \Z_n(M; \F; \Zn_2).
        \end{equation}
        Moreover, we can conclude that
        \begin{equation}
            |H^{(k)}_i|(X^{(k)}_i \times 1) \cap \mathcal{N}^i_{\eta / a_{k+1}} = \emptyset\,.
        \end{equation}
        
        \begin{figure}[!h]
            \centering
            \begin{tikzpicture}[]
                \coordinate (A1) at (1,0,0);
                \coordinate (A2) at (-1,0,0);
                \coordinate (A3) at (0,0,-{sqrt(3)});
                \coordinate (B1) at (1,1,0);
                \coordinate (B2) at (-1,1,0);
                \coordinate (B3) at (0,1,-{sqrt(3)});
                \draw (A1) -- (B1);
                \draw (A2) -- (B2);
                \draw (A3) -- (B3);
                \draw (A1) -- (A2) -- (A3) -- cycle;
                \draw [fill opacity=0.7, shade] (A2) -- (A3) -- (A1) -- node [midway] {$F_1$} cycle;

                \coordinate (a1) at (3, 0.5, 0);
                \coordinate (a2) at (4, 0.5, 0);
                \draw[->] (a1) -- node [midway, above] {$\tilde H^{(1)}$} (a2);

                \coordinate (A1') at (8,0,0);
                \coordinate (A2') at (6,0,0);
                \coordinate (A3') at (7,0,-{sqrt(3)});
                \coordinate (B1') at (8,1,0);
                \coordinate (B2') at (6,1,0);
                \coordinate (B3') at (7,1,-{sqrt(3)});
                \coordinate (C1') at (7.5, 1.5, 0);
                \coordinate (C2') at (6.5, 1.5, 0);

                \draw (A1') -- (B1');
                \draw (A2') -- (B2');
                \draw (A3') -- (B3');
                \draw (B1') -- (C1');
                \draw (B2') -- (C2');
                \draw (A1') -- (A2') -- (A3') -- cycle;
                \draw [fill opacity=0.7, shade] (A2') -- (A3') -- (A1') -- node [midway] {$F_1$} cycle;
                \draw [fill opacity=0.4, shade] (A1') -- (A2') -- (B2') -- (C2') -- (C1') -- (B1') -- cycle;
                
                \coordinate (a3) at (7, -0.5, 0);
                \coordinate (a4) at (7, -1.5, 0);
                \draw[->] (a3) -- node [midway, left] {$\cong$} (a4);

                \coordinate (a5) at (2, -0.5, 0);
                \coordinate (a6) at (5, -1.5, 0);
                \draw[->, dashed] (a5) -- node [midway, left,below] {$H^{(1)}$} (a6);
                
                \coordinate (A1'') at (8,-4,0);
                \coordinate (A2'') at (6,-4,0);
                \coordinate (A3'') at (7,-4,-{sqrt(3)});
                \coordinate (B1'') at (8,-3,0);
                \coordinate (B2'') at (6,-3,0);
                \coordinate (B3'') at (7,-3,-{sqrt(3)});

                \draw (A1'') -- (B1'');
                \draw (A2'') -- (B2'');
                \draw (A3'') -- (B3'');
                \draw (A1'') -- (A2'') -- (A3'') -- cycle;
                \draw [fill opacity=0.7, shade] (A2'') -- (A3'') -- (A1'') -- node [midway] {$F_1$} cycle;
                \draw [fill opacity=0.4, shade] (A1'') -- (A2'') -- (B2'') -- (B1'') -- cycle;
            \end{tikzpicture}
            \caption{Deformations $\tilde H^{(1)}$ and $H^{(1)}$}
        \end{figure}
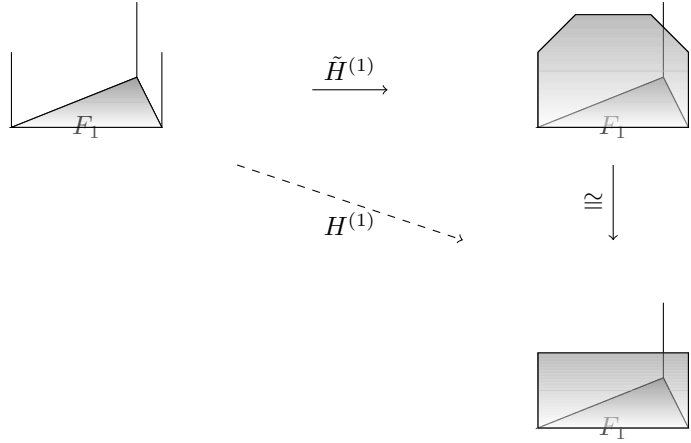
    
        In summary, we can take $\Psi_i := H^{(p)}_i(\cdot, 1)$ and $\bar\varepsilon(\Sigma):=\eta(\Sigma)/(2\cdot a_{p+1}(\Sigma))$. It follows immediately that all the conditions hold.
    \end{proof}
    
    Now, we are able to prove our first Morse index upper bound theorem.

    \begin{proof}[Proof of Theorem \ref{thm:main}]
        For any min-max sequence $\set{\Phi_i}$ for $p$-width, by Theorem \ref{Thm:deform}, there exists a new min-max sequence $\set{\Psi_i}$ such that $\mathbf{C}(\set{\Psi_i})\cap \mathcal{APR}_p \subset \mathcal{S}(\omega_p)$. Therefore, the Almgren-Pitts theory generates a minimal hypersurface with optimal regularity $V \in \mathbf{C}(\set{\Psi_i}) \cap \mathcal{APR}_p \subset \mathcal{S}(\omega_p)$, i.e., $\mathrm{index}(\mathrm{spt}(V)) \leq p$.
    \end{proof}

    In addition, combined with \cite[Theorem 4.7]{marques_morse_2018-1}, we can obtain a slightly stronger result.
    
    \begin{theorem}\label{thm:crit}
        Suppose that $(M^{n+1}, g)$ is a closed Riemannian manifold with ${n+1\geq3}$, then for any $p \in \mathbb{N}^+$, there exists a min-max sequence $\set{\Phi_i} \subset \mathcal{P}_p$ such that
        \begin{equation}
            \mathbf{C}\set{\Phi_i} \subset \mathcal{S}(\omega_p)\,.
        \end{equation}
    \end{theorem}
    \begin{proof}
        By the Deformation Theorem A (Theorem \ref{Thm:deform}) above, it suffices to show that given any pulled-tight min-max sequence $\set{\Phi_i}$ for $p$-width, we can obtain a homotopic min-max sequence $\set{\Psi_i}$ such that $\L(\set{\Psi_i}) = \L(\set{\Phi_i})$, and $\mathbf{C}(\set{\Psi_i}) \subset \mathbf{C}(\set{\Phi_i}) \cap \mathcal{APR}_p$.
    
        In \cite[Theorem~4.7]{marques_morse_2018-1}, for any fixed $j$, we can choose $R, \delta, \overline\delta$ to be  $R_j = \delta_j = \overline \delta_j = 4^{-j}$, and then we could always find a sequence $\{\Psi^j_i\}$, such that the sequence $\set{\Psi^j_i}$ satisfies all the conclusions with $\mathcal{APR}_p$ in place of $\mathcal{W}_L$, since $\mathcal{APR}_p$ is compact. Applying the diagonal method, we can obtain a desired $\set{\Psi_i}$ from $\set{\Psi^j_i}$. 
    \end{proof}

    Analogously, we can obtain a Morse index upper bound for the $\mathcal{A}^c$ functional.

    \begin{theorem}\label{thm:CMC_index}
        Given $(M^{n + 1}, g)\,(n \geq 2)$ a Riemannian manifold, $c > 0$ and $\delta > 0$, if the $(X, Z)$-homotopy class $\Pi(\Phi_0)$ of $\Phi_0: X^k \rightarrow \C(M)$ satisfies
            \begin{equation}
                \L^c(\Pi(\Phi_0)) > \sup_{x \in Z} \A^c(\Phi_0(x))\,,
            \end{equation}
        then there exists a $\Omega \in \C(M)$ such that $\A^c(\Omega) = \L^c(\Pi(\Phi_0))$ and $\partial \Omega$ is a $c$-CMC hypersurface with optimal regularity and Morse index upper bound $k$.

        In addition, the same conclusion also holds for the restrictive $(X, Z)$-homotopy class $\Pi^\delta_c(\Phi_0)$ of $\Phi_0$ for $\A^c$ with an upper bound $\delta$, provided that
            \begin{equation}
                \L^c(\Pi^\delta_c(\Phi_0)) > \sup_{x \in Z} \A^c(\Phi_0(x))\,.
            \end{equation}
    \end{theorem}

\section{Construction of \texorpdfstring{$c$}{c}-CMC hypersurfaces and improved Morse Index Bound}
    
    \subsection{Approximation by \texorpdfstring{$c$}{c}-CMC hypersurfaces}

        It is well-known that a smooth minimal hypersurface with a positive Jacobi field must be stable. In fact, it is true even for a minimal hypersurface with optimal regularity.

        \begin{lemma}\label{lem:CMC_compact}
            Suppose that $\Sigma^n \subset (M^{n + 1}, g)$ is an oriented minimal hypersurface with optimal regularity, whose unit normal in the regular part is chosen as $\nu$. If there exists a smooth positive solution $u$ to the Jacobi equation defined on $\Sigma \backslash S$, where $S \supset \mathrm{Sing}(\Sigma)$ is a closed subset of codimension no less than $7$, i.e.,
            \begin{equation}
                L_{\Sigma^n} u := - \Delta u - \left(\mathrm{Ric}^M(\nu, \nu) + |A_\Sigma|^2\right) u = 0\quad \text{in }\Sigma \backslash S\,,
            \end{equation}
            then $\Sigma$ is stable.
        \end{lemma}
        \begin{remark}
            Although $u$ is smooth in $\Sigma \backslash S$, $u$ might be unbounded near $S$.
        \end{remark}
        \begin{proof}
            Suppose not and $\Sigma$ is unstable, since $\dim S \leq n - 7$, there exists a smooth function $f$ compactly supported in $\Sigma \backslash S$ satisfying the following unstability inequality,
            \begin{equation}\label{ineq:stab}
                \int_\Sigma |\nabla f|^2 - \left(\mathrm{Ric}^M(\nu, \nu) + |A_\Sigma|^2\right) f^2 < 0\,.
            \end{equation}
            
            Note that $u$ is smooth, bounded and positive in $\mathrm{spt}\,f$, we may write $f$ as $g \cdot u$, where $g$ is compactly supported in $\Sigma \backslash S$ and thus obviously non-constant. Therefore, the left hand side is 
            \begin{equation}
                \begin{aligned}
                    &\int_\Sigma |\nabla (gu)|^2 - \left(\mathrm{Ric}^M(\nu, \nu) + |A_\Sigma|^2\right) (gu)^2\\
                    =& \int_\Sigma - \Delta(gu)\cdot gu - \left(\mathrm{Ric}^M(\nu, \nu) + |A_\Sigma|^2\right) (gu)^2\\
                    =& \int_\Sigma -\Delta g\cdot gu^2 - 2 \nabla g \cdot \nabla u \cdot gu - gu^2\left(-\Delta u - \left(\mathrm{Ric}^M(\nu, \nu) + |A_\Sigma|^2\right) u\right)\\
                    =& \int_\Sigma |\nabla g|^2 \cdot u^2 > 0\,,
                \end{aligned}
            \end{equation}
            which gives a contradiction to the inequality (\ref{ineq:stab}) above.
        \end{proof}

        In X. Zhou's multiplicity one theorem for sweepouts of boundaries (\cite[Theo\-rem~4.1]{zhou_multiplicity_2019}), he used two-sided PMC hypersurfaces to approximate minimal hypersurfaces to exclude higher-multiplicity components. Here, we will approximate minimal hypersurfaces by $c$-CMC hypersurfaces instead.

        \begin{proposition}\label{prop:compact}
            Given a sequence of two-sided, multiplicity one $c_i$-CMC hypersurfaces $V^i$ as the boundary of some Cacciopoli set with optimal regularity and Morse index upper bound $k$ in $(M, g)$, if $c_i \searrow 0$ and $\|V^i\|(M) \rightarrow A \in (0, \infty)$, then there exists a minimal hypersurface $V$ with optimal regulariy, whose support is a disjoint union of connected minimal hypersurfaces $\set{\Sigma_j}_{j=1,\cdots, l_k}$ with multiplicities $\set{m_j}$, such that, up to a subsequence,
            \begin{equation}
                V^i \rightharpoonup V\,,
            \end{equation}
            and thus $\|V\|(M) = A$.

            Moreover, every component of $V$ with multiplicity greater than $2$ is stable and
            \begin{equation}
                \sum_{m_j \leq 2} \mathrm{index}(\Sigma_j) \leq k\,.
            \end{equation}
        \end{proposition}
        \begin{proof}
            By the compactness of CMC hypersurfaces (See \cite{bellettiniCurvatureEstimatesSheeting2019}, \cite{sharp_compactness_2015}), we know that the limit is a minimal hypersurface $V$ with optimal regularity, and the support has Morse index upper bound $k$.

            It suffices to show that if $m_j \geq 3$, then the corresponding component $\Sigma_j$ is stable. For such $\Sigma_j$, we will consider two cases, i.e., either $\mathrm{reg}(\Sigma_j)$ is $2$-sided or $1$-sided.

            \textbf{Case 1}: If $\Sigma_j$ is $2$-sided, by the sheeting theorem in \cite{bellettiniCurvatureEstimatesSheeting2019} and \cite{sharp_compactness_2015}, we know that outside at most $k$ points and the singular set of $\Sigma_j$, there exists an exhaustion by compact domains $\set{U_i \subset \Sigma_j}$ and small neighborhoods $\tilde U_i$ of $U_i$ in the image of the exponential map of its normal bundle with a fixed unit normal, such that, up to a subsequence, the inverse image of $\mathrm{spt}(V^i) \cap \tilde U_i$ as a smooth multi-graph graphically converges to $\Sigma_j$. Note that the number of the graphs is determined by the multiplicity $m_j$. Therefore, $\mathrm{spt}(V^i) \cap \tilde U_i$ can be written as a set of $m_j$-normal graphs $\set{u^1_i, u^2_i, \cdots, u^{m_j}_i: u^\cdot_i \in C^\infty(U_i)}$, and
            \begin{equation}
                u^1_i \leq u^2_i \leq \cdots \leq u^{m_j}_i,
            \end{equation}
            where $u^\cdot_i \rightarrow 0$ in the smooth topology as $i \rightarrow \infty$.

            Since $V_i$ is the boundary of some Cacciopoli set, it follows immediately from the Constancy theorem that these graphs have alternate unit outer normal. In particular, $m_j \geq 3$ implies that there exists two graphs $u^1_i$ and $u^3_i$ whose unit outer normal both pointing upwards or downwards w.r.t. the unit normal over $U_i$.

            Following the proof of \cite[Theorem~4.1]{zhou_multiplicity_2019}, we can obtain a similar equation as $(4.3)$ therein, i.e.,
            \begin{equation}
                L_{\Sigma_j}(u^3_k - u^1_k) + o(u^3_k - u^1_k) = 0.
            \end{equation}
            
            Fix a point $p \in U_1$ and the Strong Maximum Principle \cite[Lemma~3.12]{zhou_multiplicity_2019} indicates that $u^3_k - u^1_k > 0$. Thus, if we define $h^k(x) := (u^3_k(x) - u^1_k(x))/(u^3_k(p) - u^1_k(p))$, a standard Harnack inequality will lead to the limit $h(x) > 0$ defined over $\Sigma_j$ outside at most $k$ points and the singular set of $\Sigma_j$, where the convergence is taken in the smooth topology. 

            In particular, $h(x)$ is a positive ``Jacobi field'' as in Lemma \ref{lem:CMC_compact}. It follows immediately that $\Sigma_j$ is stable.

            \textbf{Case 2}: If $\Sigma_j$ is $1$-sided, we can follow part 8 in the proof of \cite[Theorem~4.1]{zhou_multiplicity_2019} by considering a $2$-sided double cover $\tilde \Sigma_j$ of $\Sigma_j$. The proof in \textbf{Case 1} applied to $\tilde \Sigma_j$ implies that $\tilde \Sigma_j$ is stable and therefore $\Sigma_j$ is stable by simple lifting of vector fields.
        \end{proof}

    \subsection{Approximating \texorpdfstring{$c$}{c}-CMC hypersurfaces}
        
        For each fixed $p \in \mathbb{N}^+$, let's fix a pulled-tight min-max sequence $S = \set{\Phi_i}\subset \mathcal{P}_p$, satisfying that
        \begin{equation}
            \lim_{i\rightarrow \infty} \sup_x\set{\M(\Phi_i(x))} = \omega_p,
        \end{equation}
        and $\sup_x\set{\M(\Phi_i(x))} \leq \omega_p + 1/i$.

        Without loss of generality, we can always assume that $X_i := \mathrm{dmn}(\Phi_i)$ is of dimension $p$, since otherwise we can simply restrict $\Phi_i$ to its $p$-skeleton. Moreover, by Theorem \ref{thm:crit}, we can also assume that any $V \in \mathbf{C}(S)$ is a minimal hypersurface with optimal regularity.

        \subsubsection{Construction of Restrictive \texorpdfstring{$\F$}{F}-homotopy Families}$\,$\\

        For each fixed $i \in \mathbb{N}^+, l \in \set{0, 1, \cdots, p}$, let's define $X^{(l)}_i$ to be the $l$-skeleton of $X_i$ and  $\Phi^{(l)}_i := \Phi_i|_{X^{(l)}_i}$. The corresponding \textbf{restrictive $\F$-homotopy family} $\Pi^{(l)}_i$ for $\Phi^{(l)}_i$ is defined as
        \begin{equation}
            \begin{aligned}
                \Pi^{(l)}_{i} = \set{\Psi:X^{(l)}_i \rightarrow \Z_n(M;\F;\Zn_2)| \exists &H: X^{(l)}_i \times [0,1] \rightarrow \Z_n(M;\F;\Zn_2)\,, \\
                    & H(\cdot, 0) = \Phi^{(l)}_i\,, H(\cdot, 1) = \Psi\,,\\
                    & \sup_{x,t}\set{\M(H(x,t))} \leq \sup_x \M(\Phi_i(x)) + \frac{l}{i}}\,.
            \end{aligned}
        \end{equation}
        Thus, we can still define a min-max value $\L^{(l)}_i$ for each $\Pi^{(l)}_i$ by
        \begin{equation}
            \L^{(l)}_{i} := \inf_{\Psi \in \Pi^{(l)}_i} \sup_x \M(\Psi(x)) \leq \sup_x\M(\Phi_i(x)).
        \end{equation}

        Since $\L^{(l)}_{i}$ is well-defined for each $i$ and $l$, one can also define
        \begin{equation}
            \L^{(l)} = \liminf_{i} \L^{(l)}_{i},
        \end{equation}
        and it is easy to see that for each $l$, $\L^{(l)} \leq \omega_p$ and $\L^{(p)} = \omega_p$.

        \begin{lemma}\label{lem:L0}
            $\L^{(0)} < \omega_p$, provided that any minimal hypersurface with optimal regularity in $\mathcal{APR}_p$ is $1$-unstable.
        \end{lemma}
        \begin{proof}[Proof of Lemma \ref{lem:L0}]
            Since $S = \set{\Phi_i}$ is pulled-tight and the compact set $\mathbf{C}(S)$ consists of $1$-unstable minimal hypersurfaces with optimal regularity, we can find a finite cover $\set{\mathbf{B}^\F_{2 \varepsilon_k}(\Sigma_k)}$ from the cover $\set{\mathbf{B}^\F_{2 \varepsilon_\Sigma}(\Sigma)}_{\Sigma \in \mathbf{C}(S)}$.

            Hence, there exists a constant $\varepsilon > 0$ such that for $i$ large enough and any $x \in X^{(0)}_{i}$, either $\M(\Phi_i(x)) < \omega_p - \varepsilon$, or $\Phi_i(x)$ is in $\mathbf{B}^\F_{2 \varepsilon_k}(\Sigma_k)$ for some $k$. In the latter case, we can decrease $\M(\Phi_i(x))$ by $c_{0, k}>0$. 

            In summary, $\L^{(0)} \leq \omega_p - \min(\varepsilon, \min_k(c_{0, k}))$. 
        \end{proof}
        
        \subsubsection{Construction of Restrictive \texorpdfstring{$(X, Z)$}{(X, Z)}-homotopy Classes}$\,$\\

        In the following, we shall assume that every stable minimal hypersurface in $\mathcal{APR}_p$ is $1$-unstable, and let $l < p$ be the largest number such that $\L^{(l)} < \omega_p$. Under these assumptions, we are going to construct $c$-CMC hypersurfaces approximating a $p$-width minimal hypersurface.

        In fact, it follows immediately from the assumptions that we can take a subsequence $\set{\L^{(l)}_{i_j}}$ such that $\lim_{j \rightarrow \infty} \L^{(l)}_{i_j} = \omega_p - c_1$ for some positive $c_1$ and in addition, $\L^{(l + 1)}_{i_j} \geq \omega_p - 1/j$. For $i_j$ large enough, we have $\L^{(l)}_{i_j} < \omega_p - \frac{1}{2}c_1$, i.e., there exists a $\Psi_{i_j} \in \Pi^{(l)}_{i_j}$ with homotopy map $H_{i_j}$, such that $\sup_x \Psi_{i_j}(x) < \omega_p - \frac{1}{2}c_1$.

        Now, for each $i_j$ with $j \geq 100/ c_1$, we are going to modify $\Phi^{(l+1)}_{i_j}$ such that we can construct a restrictive $(X, Z)$-homotopy class from it. The modified map $\Psi^{(l+1)}_{i_j}$ of $\Phi^{(l+1)}_{i_j}$ is defined as follows. 

        Firstly, on the $l$-skeleton, $\Psi^{(l+1)}_{i_j}|_{X^{(l)}_{i_j}} := \Psi_{i_j}$. Then, for each $(l+1)$-cell which is a $(l+1)$-dimensional cube and thus could be parametrized by $[-1, 1]^{l+1}$. In the interior, $\Psi^{(l+1)}_{i_j}(\frac{x}{2}) := \Phi^{(l+1)}_{i_j}(x)$ for any $x \in [-1, 1]^{l+1}$. Near the boundary, $\Psi^{(l+1)}_{i_j}(\frac{t+1}{2}x) := H_{i_j}(x,t)$ for any $x \in \partial [-1, 1]^{l+1}, t \in [0,1]$. One could easily check that $\Psi^{(l+1)}_{i_j} \in \Pi^{(l+1)}_{i_j}$ and moreover, we have
        \begin{equation}
            \omega_p - 1/j \leq \sup_x \M(\Psi^{(l+1)}_{i_j}(x)) \leq \sup_x\M(\Phi_{i_j}(x))+ \frac{l}{i_j}.
        \end{equation}

        Then, since $(\Psi^{(l+1)}_{i_j})^*(\bar \lambda) = (\Phi^{(l+1)}_{i_j})^*(\bar \lambda) \neq 0$, we lift the sweepout $\Psi_{i_j}$ to the double cover $\partial: \C(M) \rightarrow \Z(M, \mathbb{Z}_2)$. We obtain a double cover $\pi: \tilde X^{(l)}_{i_j} \rightarrow X^{(l)}_{i_j}$ and the lifting map:
        \begin{equation}
            \tilde \Psi^{(l+1)}_{i_j}: \tilde X^{(l)}_{i_j} \rightarrow (\C(M),\F)\,,
        \end{equation}
        satisfying $\partial \tilde \Psi^{(l + 1)}_{i_j} = \Psi^{(l+1)}_{i_j}\comp \pi$.

        Since $\L^{(l+1)}_{i_j} \geq \omega_p - 1/j \geq \sup \M(\Psi_{i_j})|_{X^{(l)}_{i_j}} + c_1/3$ and $X^{(l+1)}_{i_j}\backslash X^{(l)}_{i_j}$ is a union of disjoint $(l+1)$-cells, there exists a $(l + 1)$-cell $\tilde C^{(l+1)}_{i_j}$ in $\tilde X^{(l+1)}_{i_j}$ satisfying the following property. Let $\Pi^\delta_c( \tilde \Psi^{(l+1)}_{i_j}|_{\tilde C^{(l+1)}_{i_j}})$ be the associated restrictive $(\tilde C^{(l+1)}_{i_j}, \partial \tilde C^{(l+1)}_{i_j})$-homotopy class, then we have
        \begin{equation}
            \begin{aligned}
                \L^{(l+1)}_{i_j} \geq \L^c(\Pi^\delta_c( \tilde \Psi^{(l+1)}_{i_j}|_{\tilde C^{(l+1)}_{i_j}})) &\geq \L^{(l+1)}_{i_j} - c\cdot \mathrm{Vol}(M)\,\\
                    &> \sup \A^c(\tilde \Psi^{(l+1)}_{i_j})|_{\partial \tilde C^{(l+1)}_{i_j}}\,,
            \end{aligned}
        \end{equation}
        provided that $c < \frac{c_1}{10000\cdot i_j \cdot \mathrm{Vol}(M)}$ and $\delta < \frac{1}{i_j}$. Indeed, if this not true, one can construct a new map $\Psi' \in \Pi^{(l + 1)}_{i_j}$ such that $\sup \mathbf{M}(\Psi) < \L^{(l+1)}_{i_j}$, which gives a contradiction to the definition of $\L^{(l+1)}_{i_j}$.

        \subsubsection{Construction of \texorpdfstring{$c$}{c}-CMC hypersurfaces}$\,$\\

        For each $i_j$ with $j \geq 100/ c_1$, applying Theorem \ref{thm:CMC_index} on $\Pi^\delta_c( \tilde \Psi^{(l+1)}_{i_j}|_{\tilde C^{(l+1)}_{i_j}})$, we obtain a $c$-CMC minimal hypersurface $V_{i_j}$ with Morse index no greater than $p$ with $c = \frac{c_1}{100000\cdot i_j \cdot \mathrm{Vol}(M)}$ and $\delta = \frac{1}{10\cdot i_j}$. Moreover,
        \begin{equation}
            \A^c(V_{i_j}) \leq \|V_i\|(M) \leq \A^c(V_{i_j}) + c \cdot \mathrm{Vol}(M)\,.
        \end{equation}
        and thus, by the assumption of $c$, we have
        \begin{equation}
            \L^{(l+1)}_{i_j} -  \frac{c_1}{100000\cdot i_j} \leq \|V_i\|(M) \leq \L^{(l+1)}_{i_j} + \frac{c_1}{100000\cdot i_j}\,.
        \end{equation}

    \subsection{Proof of improved Morse Index Bound}\quad

        If there exists a stable minimal hypersurface with optimal regularity in $\mathcal{APR}_p$, then the conclusion holds apparently.

        Otherwise, by the construciton in the previous subsection, we know that there exists a sequence of $c_i$-CMC hypersurfaces $V_i$ with optimal regularity and Morse index upper bound $p$ satisfying that
        \begin{align}
            \|V_i\|(M) &\rightarrow \omega_p\,,\\
            c_i &\searrow 0\,.
        \end{align}
        It follows immediately from Proposition \ref{prop:compact} that there exists a $p$-width minimal hypersurface $V$ with optimal regularity, whose support is a disjoint union of connected minimal hypersurfaces $\set{\Sigma_j}_{j=1,\cdots, l_k}$ with multiplicities $\set{m_j}$ such that every component of $V$ with multiplicity greater than $2$ is stable and
            \begin{equation}
                \sum_{m_j \leq 2} \mathrm{index}(\Sigma_j) \leq p\,.
            \end{equation}

\bibliography{reference}
\end{document}